\documentclass[10pt,twoside,english,reqno,a4paper]{amsart}
\usepackage{listings,graphicx,amsmath,varioref,amscd,amssymb,color,bm,stmaryrd,amsthm,amsfonts,graphics,geometry,latexsym,pgf,pst-all} 
\theoremstyle{plain}
\usepackage{esint}
\usepackage{amsthm}

\theoremstyle{plain}
\newtheorem{theorem}{Theorem}[section]

\newtheorem{lemma}[theorem]{Lemma}

\usepackage{geometry}
\usepackage[colorlinks=false]{hyperref}

\theoremstyle{definition}
\newtheorem{defin}[theorem]{Definition}

\newtheorem{remark}[theorem]{Remark}
\newtheorem{example}{Example}

\theoremstyle{remark}

\usepackage{fouriernc}
\usepackage[T1]{fontenc}

\def\bk{\color{black}}

\numberwithin{equation}{section}

\def\dis{\displaystyle}

\DeclareMathOperator{\diver}{div}

\DeclareMathOperator{\R}{\mathbb{R}}

\newcommand{\car}[1]{\raise1pt\hbox{$\chi$}_{#1}}

\newcommand{\DM }{\mathcal{DM}^\infty }
\def\rn{\mathbb{R}^N}

\newcommand{\res}{\!\!\mathop{\hbox{
			\vrule height 7pt width .5pt depth 0pt
			\vrule height .5pt width 6pt depth 0pt}}
	\nolimits}

\begin{document}
\title{1-Laplacian type problems  with  strongly singular nonlinearities and gradient terms}

\author[D. Giachetti]{Daniela Giachetti}
\author[F. Oliva]{Francescantonio Oliva}
\author[F. Petitta]{Francesco Petitta}

\address{Dipartimento di Scienze di Base e Applicate per l' Ingegneria, Sapienza Universit\`a di Roma, Via Scarpa 16, 00161 Roma, Italia}
\address[Daniela Giachetti]{daniela.giachetti@sbai.uniroma1.it}
\address[Francescantonio Oliva]{francesco.oliva@sbai.uniroma1.it}
\address[Francesco Petitta]{francesco.petitta@sbai.uniroma1.it}

\keywords{1-Laplacian, Nonlinear elliptic equations, Singular elliptic equations, Gradient terms} \subjclass[2010]{35J25, 35J60,  35J75, 35R99, 35A01}

\begin{abstract}
We show optimal existence, nonexistence and regularity results for nonnegative solutions to Dirichlet problems as 
$$
\begin{cases}
\dis -\Delta_1 u = g(u)|D u|+h(u)f & \text{in}\;\Omega,\\
u=0 & \text{on}\;\partial\Omega,
\end{cases}
$$
where $\Omega$ is an open bounded subset of $\mathbb{R}^N$,  $f\geq 0$ belongs to $L^N(\Omega)$, and $g$ and $h$ are continuous functions  that may blow up at zero.  As a noteworthy fact we show how a non-trivial interaction mechanism between the two nonlinearities $g$ and $h$ produces remarkable regularizing effects on the solutions. \bk
The sharpness of our main results is discussed through the use of appropriate explicit examples. 
\end{abstract}

\maketitle

\tableofcontents

\section{Introduction}

We consider homogeneous Dirichlet problems as 
 \begin{equation} \label{pb1lapintro}
				 \begin{cases}
					 \displaystyle -\Delta_1 u = g(u)|D u|+ h(u)f &  \text{in}\, \Omega, \\
					 u\geq 0 &  \text{in}\, \Omega,\\
					 u=0 & \text{on}\ \partial \Omega,			
				 \end{cases}
			 \end{equation}
where $\Omega$ is a bounded open subset of $\R^N$ with Lipschitz boundary,  $f\in L^{N}(\Omega)$ is a nonnegative function, $g(s)$ and $h(s)$ are  nonnegative continuous functions defined on $[0,\infty)$,  and  possibly singular at $s=0$ (i.e. $g(0)=\infty$ and/or $h(0)=\infty$). 

  Here $\Delta_1 u$ is the formal limit of the $p$-laplace operator as $p\to1^+$; i.e. $\Delta_1 u = {\rm div} (Du/|Du|)$.   The natural space to set  this kind of problems is  $BV$ (or its local version $BV_{\rm loc}$), the space of functions of bounded variation, i.e. the space of $L^1$ functions whose gradient is a Radon measure with finite (or locally finite) total variation. The ratio appearing in the definition of the $1$-laplace operator   $\frac{D u}{|D u|}$ has to be interpreted as  the Radon-Nikodym derivative of the measure $Du$ with respect to its total variation $|Du|$.  Two  of the most striking  differences with the $p>1$ case rely in the non-compactness of the traces (the boundary datum needs not to be attained point-wise) and a structural non-uniqueness phenomenon based on the homogeneity of the operator.

\medskip

If one considers the  autonomous case without gradient terms (i.e. $h\equiv 1$ and $g\equiv 0$),  problems involving the $1$-laplace   operator  arise in  the study of   image restoration as well as in  torsion problems (\cite{K, ka, Sapiro, M, BCRS}).  The non-autonomous and non-singular case (again with $g\equiv 0$) has also been considered in frameworks of   more theoretic nature as eigenvalues problems and  critical Sobolev exponent  (see \cite{KS, D} and references therein).   Also, $1$-laplace type operators  are   known to be closely related to the mean curvature operator (\cite{OsSe}); in fact, as the unit normal  of the level set $\{u(x) = k\}$ is given formally by $n(x) = Du/|Du|$, then  the mean curvature of this surface at the point $x$ is formally given by $$H(x) = {\rm div}(n)(x) = {\rm div}(Du/|Du|)(x)\,;$$   this relationship clearly expresses  that the behavior at the boundary $\partial\Omega$ of  solutions to problems as in \eqref{pb1lapintro} may depend on the geometry of the boundary.

  Equations with dependence on the gradient  also enter in geometric problems as the one proposed  in \cite{HL} in the study of the inverse mean curvature flow (see also \cite{mazonsegura}).  \bk 
We refer the interested reader to the monograph \cite{ACM} for a more complete review on applications.

\medskip

The case of a possibly singular nonlinearity $h$  in \eqref{pb1lapintro} with   $g=0$ has been studied, in the case of a $p$-laplace leading term  with $p>1$, in connection with the analysis of  flows of  non-Newtonian fluid as the  pseudoplastic ones; these kinds of equations appear in particular in geophysical phenomena (e.g. glacial advance) as well as in industrial applications as extrusion in polymers or metals.  We refer to \cite[Section 3]{DG} for a detailed derivation of the model in the case $p=2$. The mathematical literature in this case is  massive; without the aim to be complete we refer the reader to \cite{crt, LM,BO,OP1,OP, GMM,GMM2,DCA,O} and references therein. 

\medskip

From the purely mathematical point of view, the case $p=1$ is faced in \cite{CT,MST1, MST2} in the autonomous case by approximating  solutions to 
\begin{equation}\label{remar}\begin{cases}
					 \displaystyle - \Delta_{1} v = f &  \text{in}\, \Omega, \\
					 v=0 & \text{on}\ \partial \Omega\,,		
				 \end{cases}
\end{equation}
with solutions $v_p$ to the associated $p$-laplacian problems with $p>1$.  This procedure presents remarkable features;  first of all a degeneracy appears in the approximation argument if the datum is too small, say $\|f\|_{N}<{\mathcal{S}_{1}}^{-1}$,  $\mathcal{S}_{1}$ being the best Sobolev constant in $W^{1,1} (\R^N)$; in this case,  $ v_{p}\to 0$ a.e. on $ \Omega$. On the other hand the approximating solutions $v_p$ may blow up on a set of positive Lebesgue  measure if $f$ does not belong to $L^N(\Omega)$. 

Concerning the presence of a possibly singular $h$, in  \cite{DGOP} existence and  regularity of a nonnegative  (nontrivial, in general)  distributional solutions to the homogeneous Dirichlet problem  
\begin{equation}\label{introapp}\begin{cases}
					 \displaystyle - \Delta_{1} v = h(v)f &  \text{in}\, \Omega, \\
					 v=0 & \text{on}\ \partial \Omega\,,		
				 \end{cases}
\end{equation}
is obtained for a   nonnegative datum $f$ in $L^{N}(\Omega)$ with suitable small norm.  Uniqueness of solutions is also derived  provided  $h$ is decreasing and $f>0$. 
\bk

\medskip

The situation significantly changes when one looks at the case $g\not= 0$, that is the case of a gradient term that depends on the solution itself with natural growth. If $p>1$ and  $h\equiv 1$, then problems as  
\begin{equation}\label{hu1}
			 \begin{cases}
					 \displaystyle -\Delta_p w = g(w)|\nabla w|^p+ f &  \text{in}\, \Omega, \\
					 w=0 & \text{on}\ \partial \Omega\,;		
				 \end{cases}
\end{equation}
as regards  the non-singular case (i.e. with a bounded continuous $g$) one can refer to \cite{bmp,bmp2, fm, gt, ps} for a companion on the subject. The possibly singular case have been  largely investigated both in the absorption and in the reaction case. If  $p=2$ and  $g(s)\sim s^{-\theta}$ one may refer to \cite{b1,a6, ABLP, gps, GPS2} and references therein, while the case $p>1$ has also been considered (\cite{zw, ww, dop}). Observe that, in any cases, the threshold $\theta=1$ is shown to be critical in order to get global  finite energy solutions for a general nonnegative datum $f$ (see also the discussion in \cite{d}).

The case $p=1$ of problem  \eqref{hu1} has been recently faced  mostly in presence of an absorption term; in \cite{mazonsegura, fp} ($g\equiv -1$) and \cite{ls} (bounded negative $g$). The reaction case  is studied in \cite{ads} for $g\equiv 1$ and in presence of zero order absorption term (see also \cite{ds2}). 

\bk

\medskip

In this paper we extend the previous results in many directions; under very general assumptions on the data, we show optimal existence and regularity results for solutions of problem \eqref{pb1lapintro}. As predictable, the goal will be accomplished by mean of an approximation argument with $p$-laplace type problems
\begin{equation}\label{pp}
				 \begin{cases}
					 \displaystyle -\Delta_p u_p = g_p(u_p)|\nabla u_p|^p+ h_p(u_p)f &  \text{in}\, \Omega, \\
					 u_p=0 & \text{on}\ \partial \Omega,			
				 \end{cases}
\end{equation}
where $p>1$ and $g_p$ and $h_p$ are suitable truncations of respectively  $g$ and $h$. Our results  wholly agree with the existing literature and, as proper examples will show, they  are sharp in the sense outlined later on.  For instance, the sharp smallness assumption of \cite{DGOP} for problem \eqref{introapp} is recovered,  as well as the results of \cite{CT, MST1} for \eqref{remar}. Furthermore, if $g=h\equiv 1$  in \eqref{pb1lapintro} one also recovers the existing results  (see, for instance, \cite{ds2} and references therein).  

\medskip 
One of the main difficulties, of course, will rely on carefully keeping track of the involved constants in order to get  (sharp)  a priori estimates that  do not depend on the parameter $p$. Here is where a smallness  assumption on the data will be needed. A crucial point regards the proof that the candidate solution  $u$ is bounded and it shall be achieved by mean of a comparison with  suitable approximating solutions of problem \eqref{introapp}. 
Then, after that,  in order to pass to the limit in  \eqref{pp}  and to show that the candidate $u$ is actually a solution to \eqref{pb1lapintro},  a suitable chain rule formula will be established.   A key ingredient in order to conclude will rely on the proof that the jump part of the derivative of  $u$ is zero; this peculiar phenomenon is  due to the presence of the gradient term in the equations and it does not occur in the case $g=0$.  
\bk

A further drawback that has to be dealt with  concerns the way the boundary datum is assumed.  In the non-singular case this is quite clear nowadays and, as we already mentioned, a weak boundary requirement is needed as no point-wise behavior can be prescribed; here,  due to the presence of the possibly singular nonlinearity, most of the estimates one shall find are only local and one needs to further clarify the notion of the homogeneous  boundary datum. A general result on vector fields whose divergence is a nonnegative  (local)  Radon measure will be used for this purpose. 

\medskip

One of our concurrent  purposes consists in showing how the two nonlinearities interact  with each other and how they sort-of regularize the problem. We already mentioned how the presence of the nonlinear gradient term gives rise to a solution without jump part.  Moreover,  the term $g(u)|Du|$ is, a priori, only a (locally) bounded measure; despite this,  the approximation argument  do find a solution which is finite a.e. on $\Omega$ in contrast with the blow up behavior of the solutions $v_p$ approximating \eqref{remar}. 

Another regularizing effect appears if $h(0)=\infty$ as no degeneracy of the approximating sequence $v_p$ is produced;  in some sense the behavior of $h$ near zero compensates the possibly small norm of $f$ in such a way to return a positive limit solution $u$.

\medskip \medskip

The plan of the paper is the following:  in Section \ref{due} after  providing some  basic notations on $BV$ spaces,  we set the  Anzellotti-Chen-Frid type theory of vector fields with measure-valued divergence we use; in particular, we prove a general property on those vector fields whose divergence is a signed measure (Lemma \ref{lemmal1}).  We also prove a useful generalized Chain rule formula (Lemma \ref{chainrule}). For the sake of exposition Section \ref{sec3} will be devoted to the case of a positive datum $f$ and a subcritical nonlinearity $g$ (i.e. $g(s)\sim s^{-\theta}$, with $\theta<1$). Here the approximation scheme is introduced and some properties of an  auxiliary problem (namely problem \eqref{pp} with $g=0$) is presented.  This paves the way to the proofs of the basic a priori estimates  we need and of the boundedness of the candidate solution $u$.  In Section \ref{3.3} we show that $u$ has no jump part and (Section \ref{3.4}) we pass to the limit in \eqref{pp}. In Section \ref{strongly} we extend the result to the critical case $\theta =1$,  while Section \ref{nonnegative} is focused on the case of a general nonnegative datum $f$. Finally Section \ref{rae} is devoted to discuss some further remarks and examples. We first show how the case $p=1$ and $\theta=1$ behaves differently with respect to  the case $p>1$ showing an intriguing breaking of the nonexistence threshold. Two examples are then built in order to show how the geometry of $\Omega$ influences the behavior of the solutions of our problem; in particular, at least in some model cases, constant solutions are shown to exist if $\Omega$ is {calibrable} while the existence of non-constant solutions is proven once the {variational mean curvature} of $\Omega$ is not bounded.  

\subsection*{Notation}

We denote by $\mathcal H^{N-1}(E)$ the $(N - 1)$-dimensional Hausdorff measure of a set $E$ while $|E|$ stands for its $N$-dimensional  Lebesgue measure. 
 $\mathcal{M}(\Omega)$ is the usual space of Radon measures with finite total variation over $\Omega$. The space $\mathcal{M}_{\rm loc}(\Omega)$ is the space of Radon measure which are locally finite in $\Omega$. We refer to a Lebesgue space with respect to a Radon measure $\mu$ as $L^q(\Omega,\mu)$.  We denote by $\mathcal{S}_1$ the best constant in the Sobolev inequality:  i.e.
$$||v||_{L^{\frac{N}{N-1}}(\Omega)} \le \mathcal{S}_1 ||v||_{W^{1,1}_0(\Omega)}, \ \ \forall v\in W^{1,1}_0(\Omega).$$

\medskip 

 We denote by $\chi_{E}$ the characteristic function of a set $E$. For a fixed $k>0$, we use the truncation functions $T_{k}:\R\to\R$ and $G_{k}:\R\to\R$ defined by
\begin{align*}
T_k(s):=&\max (-k,\min (s,k))\ \ \text{\rm and} \ \ G_k(s):=s- T_k(s).
\end{align*}
 \begin{figure}[htbp]\centering
\includegraphics[width=2in]{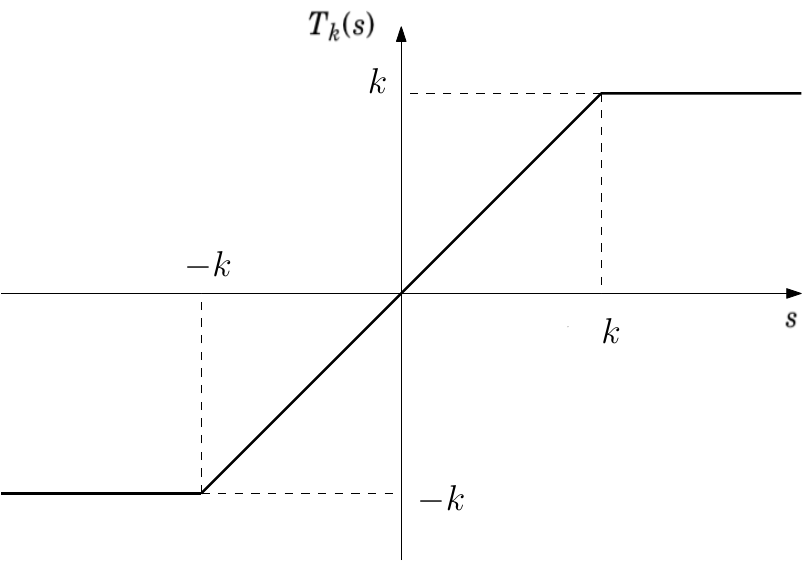} \ \ \ \includegraphics[width=2in]{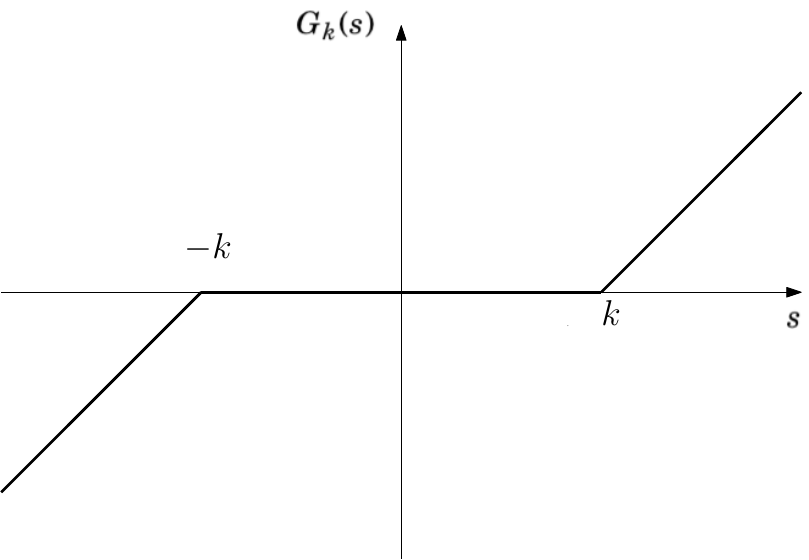}
\caption{$T_k(s)$ and $G_k(s)$}
\end{figure}

We also use the following auxiliary function defined for nonnegative values
\begin{align}\label{Vdelta}
\displaystyle
V_{\delta}(s):=
\begin{cases}
1 \ \ &  0\le s\le \delta, \\
\displaystyle\frac{2\delta-s}{\delta} \ \ &\delta <s< 2\delta, \\
0 \ \ &s\ge 2\delta.
\end{cases}
\end{align}
 \begin{figure}[htbp]\centering
\includegraphics[width=2in]{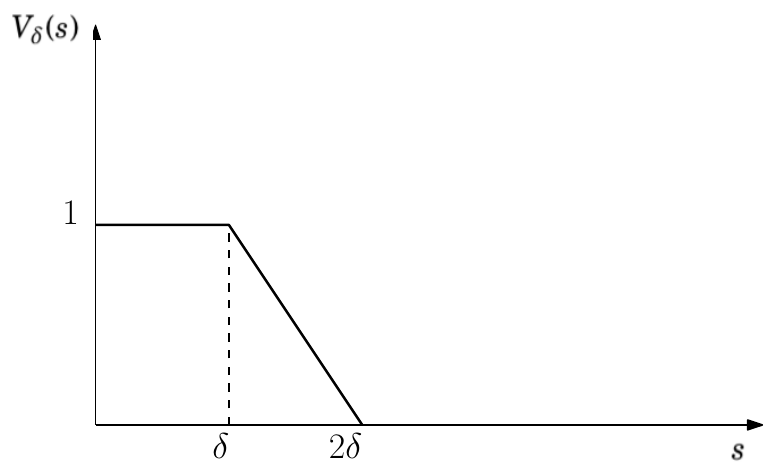}
\caption{$V_\delta (s)$}
\end{figure}

\medskip 

If no otherwise specified, we will denote by $C$ several positive constants whose value may change from line to line and, sometimes, on the same line. These values will only depend on the data but they will never depend on the indexes of the sequences we will gradually introduce.

Finally for simplicity's sake, and if there is no ambiguity,  we will often use the following notation for the Lebesgue integral of a funciton $f$ 
$$
 \int_\Omega f:=\int_\Omega f(x)\ dx\,,
$$
or

$$
 \int_\Omega f\mu:=\int_\Omega f(x)\ d\mu \,,
$$
to indicate the integration of a function $f$ with respect to a measure $\mu$. 
\section{Preliminaries}  \label{due}
\subsection{Basics on $BV$ spaces} 
We briefly review some basic facts on  $BV$ spaces;  we refer to \cite{afp} for a complete account  and for further standard notations not included here for the sake of brevity. 
Let $\Omega$ be an open bounded subset of $\R^N$ ($N> 1$) with Lipschitz boundary. The space of functions with bounded variation on $\Omega$ is defined as follows:  
$$BV(\Omega):=\{ u\in L^1(\Omega) : Du \in \mathcal{M}(\Omega)^N \}.$$ 
We underline that the $BV(\Omega)$ space endowed with the norm  
$$ ||u||_{BV(\Omega)}=\int_\Omega |u|\, + \int_\Omega|Du|\,,$$
or with
$$\displaystyle ||u||_{BV(\Omega)}=\int_{\partial\Omega}
|u|\, d\mathcal H^{N-1}+ \int_\Omega|Du|,$$
is a Banach space. We denote by $BV_{\rm loc}(\Omega)$ the space of functions in $BV(\omega)$ for every open set $\omega \subset\subset\Omega$. With $L_u$ we denote the set of Lebesgue point of a function $u$, with $S_u =\Omega\setminus L_u$ and  with $J_u$ the  jump set. It is well known that any function $u\in BV(\Omega)$ can be identified with its precise representative $u^*$ which is the Lebesgue representative in $L_u$ while $u^*=\frac{u^++u^-}{2}$ in $J_u$ where $u^+,u^-$ are the approximate limits of $u$. Moreover it can be shown that 
$\mathcal{H}^{N-1}(S_u\setminus J_u)=0$ and that $u^*$ is well defined $\mathcal{H}^{N-1}$-a.e.

\subsection{The Anzellotti-Chen-Frid theory.} In order to be self-contained we  summarize the $L^\infty$-divergence-measure vector fields theory due to \cite{An} and \cite{CF}. We denote by 
$$\DM(\Omega):=\{ z\in L^\infty(\Omega)^N : \operatorname{div}z \in \mathcal{M}(\Omega) \},$$
and by $\DM_{\rm loc}(\Omega)$ its local version, namely the space of bounded vector field $z$ with $\operatorname{div}z \in \mathcal{M}_{\rm loc}(\Omega)$.
We first recall that if $z\in \DM(\Omega)$ then $\operatorname{div}z $ is an absolutely continuous measure with respect to $\mathcal H^{N-1}$. 
\\In \cite{An} the following distribution $(z,Dv): C^1_c(\Omega)\to \mathbb{R}$ is considered: 
\begin{equation}\label{dist1}
\langle(z,Dv),\varphi\rangle:=-\int_\Omega v^*\varphi\operatorname{div}z-\int_\Omega
vz\cdot\nabla\varphi,\quad \varphi\in C_c^1(\Omega).
\end{equation}
In \cite{MST2} and \cite{C} the authors prove that $(z, Dv)$ is well defined if $z\in \DM(\Omega)$ and $v\in BV(\Omega)\cap L^\infty(\Omega)$ since one can show that $v^*\in L^\infty(\Omega,\operatorname{div}z)$. Moreover in \cite{dgs} the authors show that \eqref{dist1} is well posed if $z\in \DM_{\rm loc}(\Omega)$ and $v\in BV_{\rm loc}(\Omega)\cap L^1_{\rm loc}(\Omega, \operatorname{div}z)$ and it holds that
\begin{equation*}\label{finitetotal}
|\langle   (z, Dv), \varphi\rangle| \le ||\varphi||_{L^{\infty}(U) } ||z||_{L^\infty(U)^N} \int_{U} |Dv|\,,
\end{equation*}
for all open set $U \subset\subset \Omega$ and for all $\varphi\in C_c^1(U)$. Moreover one has
\begin{equation*}\label{finitetotal1}
\left| \int_B (z, Dv) \right|  \le  \int_B \left|(z, Dv)\right| \le  ||z||_{L^\infty(U)^N} \int_{B} |Dv|\,,
\end{equation*}
for all Borel sets $B$ and for all open sets $U$ such that $B\subset U \subset \Omega$.

Observe that,  if $z\in \DM_{\rm loc}(\Omega)$ and $w\in BV_{\rm loc}(\Omega)\cap L^\infty(\Omega)$, then
$$
\diver (wz)=(z,Dw) + w^*\diver z\,,
$$
so that $wz \in \DM_{\rm loc}(\Omega)$. This fact allows us to check the following technical  result that will be useful in the sequel and that contains a particular instance of \cite[Lemma 2.6]{gmp}

\begin{lemma}\label{l-ms}
	Let $z\in \DM_{\rm loc}(\Omega)$ and let $ u,v\in BV_{\rm loc} (\Omega)\cap L^\infty(\Omega) $  such that $\mathcal{H}^{N-1}(J_u\cap J_v)=0$. Then 
	\begin{equation}
		\label{f1}
		(z,D(uv))={ u^* \bk}(z,Dv)+(vz,Du). 
	\end{equation}
\end{lemma}
\begin{proof}   Since $\mathcal{H}^{N-1}(J_u\cap J_v)=0$, one has $(uv)^*=u^*v^*$ $\mathcal{H}^{N-1}$-a.e. in $\Omega$ (and hence $|\operatorname{div} z|$-a.e. in $\Omega$).  A reiterated application of \eqref{dist1} gives that, as measures
	\begin{eqnarray*}
		(z,D(uv)) & = & -{ (uv)^* }\operatorname{div} z+\operatorname{div}(uvz) { = -u^*v^*\operatorname{div} z+\operatorname{div}(uvz)}
		\\ &=& -{ u^* }(\operatorname{div}(vz)-(z,Dv)) + { u^* }\operatorname{div}(vz) + (vz,Du)
		\\ &=& { u^* }(z,Dv) + (vz,Du)
	\end{eqnarray*}
	and \eqref{f1} is proven.
\end{proof}

\medskip

We recall that in \cite{An} it is proved that every $z \in \mathcal{DM}^{\infty}(\Omega)$ possesses a weak trace on $\partial \Omega$ of its normal component which is denoted by
$[z, \nu]$, where $\nu(x)$ is the outward normal unit vector defined for $\mathcal H^{N-1}$-almost every $x\in\partial\Omega$. Moreover, it holds
\begin{equation*}\label{des1}
||[z,\nu]||_{L^\infty(\partial\Omega)}\le ||z||_{L^\infty(\Omega)^N},
\end{equation*}
and it also satisfies that if $z \in \mathcal{DM}^{\infty}(\Omega)$ and $v\in BV(\Omega)\cap L^\infty(\Omega)$, then
\begin{equation*}\label{des2}
v[z,\nu]=[vz,\nu],
\end{equation*}
(see \cite{C}).\\
Finally we will also use the following Green formula due to \cite{dgs}, the authors prove that if $z\in \DM_{\rm loc}(\Omega)$ and $v\in BV(\Omega)\cap L^\infty(\Omega)$ such that $v^*\in L^1(\Omega,\operatorname{div}z)$ then $vz\in \DM(\Omega)$ and a weak trace can be defined as well as the following Green formula:
\begin{lemma}
	Let $z\in \DM_{\rm}(\Omega)$ and let $v\in BV(\Omega)\cap L^\infty(\Omega)$ then it holds
	\begin{equation}\label{green}
	\int_{\Omega} v^* \operatorname{div}z + \int_{\Omega} (z, Dv) = \int_{\partial \Omega} [vz, \nu] \ d\mathcal H^{N-1}.
	\end{equation}	
\end{lemma}

\medskip

We also have the following result that extends \cite[Lemma 5.3]{DGOP} to the measure case and that has its own interest. 
	\begin{lemma}\label{lemmal1}
	Let $0\le \mu \in \mathcal{M}_{\rm loc}(\Omega)$  and let $z\in \mathcal{D}\mathcal{M}^\infty_{\rm{loc}}(\Omega)$  such that 
	\begin{equation}\label{lemma_distr*}
	-\operatorname{div}z = \mu \ \text{    in    } \mathcal{D'}(\Omega),
	\end{equation}
	then 
$$
	\mu \in \mathcal{M}(\Omega).
$$
\end{lemma}
\begin{proof}
	Let $0\le v\in W^{1,1}_0(\Omega)\cap C(\Omega)$ and let $\varphi_n\in C^{1}_c(\Omega)$ be a sequence of nonnegative functions converging to $v$ in $W^{1,1}_0(\Omega)$. Let us take $\varphi_n$ as test function in \eqref{lemma_distr*} 
	\begin{equation*}\label{contest}
	\int_\Omega z \cdot \nabla \varphi_n = \int_\Omega  \varphi_n \mu
	\end{equation*}
	and we take $n\to \infty$ by applying the Fatou Lemma on the right hand side of the previous. Hence one obtains 
	\begin{equation}\label{contestbis}
	\int_\Omega  v \mu\le \int_\Omega z \cdot \nabla v.
	\end{equation}	
	Now we take $\tilde{v} \in W^{1,1}(\Omega)\cap C(\Omega)$ and then from a Gagliardo Lemma (see \cite[Lemma 5.5]{An}) there exists $w_n\in
	W^{1, 1}(\Omega)\cap C(\Omega)$ having $|w_n|_{\partial\Omega}=|\tilde{v}|_{\partial\Omega}$, $\displaystyle\int_\Omega|\nabla
	w_n|\le\displaystyle\int_{\partial\Omega}\tilde{v}\,d\mathcal H^{N-1}+\frac1n\,$ and such that $w_n$ tends to $0$ everywhere in $\Omega$.
	Now take $|v-w_n|\in W_0^{1,1}(\Omega)\cap C(\Omega)$ as a test function in \eqref{contestbis}, yielding to
	\begin{align*}
	\int_\Omega |\tilde{v}-w_n|\mu &\le  \int_\Omega z\cdot\nabla|\tilde{v}-w_n| \le ||z||_{L^\infty(\Omega)^N}\left(\int_\Omega|\nabla \tilde{v}| + \int_\Omega|\nabla w_n| \right)
	\\
	&\le  ||z||_{L^\infty(\Omega)^N}\left(\int_\Omega|\nabla \tilde{v}|+\int_{\partial\Omega}\tilde{v}\,d\mathcal H^{N-1}+\frac1n\right),
	\end{align*}
	and the Fatou Lemma with respect to $n$ gives
	\begin{equation*}
	\int_\Omega \tilde{v}\mu\le  ||z||_{L^\infty(\Omega)^N}\left(\int_\Omega|\nabla \tilde{v}|+\int_{\partial\Omega}\tilde{v}\,d\mathcal H^{N-1}\right),
	\end{equation*}
	where, taking $\tilde{v}\equiv 1$, one deduces that $\mu$ belongs to $\mathcal{M}(\Omega)$. 
\end{proof}

\medskip

\subsection{A chain rule formula.} We will also use this type of chain rule formula which is an extension of the classical one (see   \cite[Theorem 3.99]{afp}) for functions having unbounded derivative at zero. The following lemma also extends  \cite[Theorem 3.10]{ls}. 

\begin{lemma}\label{chainrule}
	Let $0\le u\in BV_{\rm loc}(\Omega)$ with $D^j u=0$ and let $g:[0,\infty)\mapsto (0,\infty]$ be a continuous function finite outside $s=0$. Then define $\Gamma:[0,\infty)\mapsto [-\infty,+\infty)$ such that $\Gamma'(s)=g(s)$ and suppose that $v=\Gamma(u)\in BV_{\rm loc}(\Omega)$. Then $g(u)\chi_{\{u>0\}}|Du|$ is a locally finite measure and it holds 
	\begin{equation}\label{fcl}
	\chi_{\{u>0\}} |Dv|=g(u)\chi_{\{u>0\}}|Du| \text{ as measures in } \Omega\,.
	\end{equation}
\end{lemma}
\begin{proof}
	We consider a Lipschitz function $\Gamma_{\epsilon}(s)$ whose values agree with $\Gamma(s)$ for $\epsilon<s<\frac{1}{\epsilon}$, namely
	$$
	\Gamma_{\epsilon}(s)=
	\begin{cases}
	\Gamma(\epsilon), \ &s\le \epsilon,\\
	\Gamma(s), \ & \epsilon<s<\frac{1}{\epsilon},\\
	\Gamma\left(\frac{1}{\epsilon}\right), \ &s\ge \frac{1}{\epsilon}.	
	\end{cases}
	$$
	As the function $\Gamma_{\epsilon}(s)$ is Lipschitz then one can apply the classical chain rule to deduce for a nonnegative $\varphi\in C^1_c(\Omega)$ that
	\begin{equation*}
	\int_{\{\epsilon < u<\frac{1}{\epsilon}\}} \varphi |Dv| =  \int_{\{\epsilon < u<\frac{1}{\epsilon}\}} \varphi |D\Gamma(u)| = \int_{\{\epsilon < u<\frac{1}{\epsilon}\}} \varphi |D\Gamma_\epsilon(u)|= \int_{\{\epsilon < u<\frac{1}{\epsilon}\}} \varphi g(u)|Du|,
	\end{equation*} 
	where the second equality holds by the means of Proposition $3.92$ of \cite{afp}. Indeed, one has that $D (\Gamma (u) -\Gamma_\epsilon (u))= 0$ (and so $|D \Gamma (u) |=|D\Gamma_\epsilon (u)|$) on $\{\epsilon < u<\frac{1}{\epsilon}\}$. 
	 Now taking $\epsilon \to 0$ it follows from the monotone convergence Theorem that
	\begin{equation*}\label{chainrulepositive}
	\int_{\{u>0\}} \varphi |Dv| = \int_{\{u>0\}} \varphi g(u)|Du|\,,
	\end{equation*}	
and the proof is concluded.
\end{proof}
\begin{remark}\label{remarkchain}
We presented the result of Lemma \ref{chainrule} in general form; however  it is worth mentioning that if $\Gamma (0)=0$ then \eqref{fcl} becomes
$$
	|Dv|=g(u)\chi_{\{u>0\}}|Du| \text{ as measures in } \Omega\,
$$
as, using \cite[Proposition 3.92]{afp}, one has
$$
\int_{\{u>0\}} \varphi |Dv| =\int_\Omega \varphi |Dv|\,.
$$
Let us point out that this will always be the case for us  in the sequel besides some instances in which, as a matter of fact,  $u>0$ and so \eqref{fcl} simply reduces to $|Dv|=g(u)|Du|$. 
\end{remark}

\section{Main assumptions and results}  \label{sec3}
Let   $\Omega$  be  an open bounded subset of $\R^N$ ($N> 1$) with Lipschitz boundary and let us  consider  the following Dirichlet problem 
\begin{equation}
\label{pb}
\begin{cases}
\dis -\Delta_1 u = g(u)|D u|+h(u)f & \text{in}\;\Omega,\\
u=0 & \text{on}\;\partial\Omega,
\end{cases}
\end{equation}
where $-\Delta_1$ is the so-called 1-Laplace operator and  the datum $f$ is a nonnegative function belonging  to $L^N(\Omega)$.  
The nonlinearities $g:[0,\infty)\mapsto (0,\infty]$ and $h:[0,\infty)\mapsto [0,\infty]$ are assumed to be merely continuous and finite outside the origin;   we require the following controls near zero
\begin{equation}\label{g1}
\displaystyle \exists\;{c_1},\theta,s_1>0\;\ \text{such that}\;\  g(s)\le \frac{c_1}{s^\theta} \ \ \text{if} \ \ s\leq s_1, 
\end{equation}
and 
\begin{equation}\label{h1}
\begin{aligned}
\displaystyle &\exists\;{c_2},\gamma,s_2>0\;\ \text{such that}\;\  h(s)\le \frac{c_2}{s^\gamma} \ \ \text{if} \ \ s\leq s_2, \\
&\lim_{s\to \infty} h(s):=h(\infty)<\infty. 
\end{aligned}
\end{equation}
We stress that under the above assumptions the case of  $g, h$ being bounded is allowed. For the sake of presentation,  in this section,  we state and prove the existence of a solution to \eqref{pb} in the milder singular case \begin{equation}\label{mild}0<\theta<1\ \ \text{and}\ \ \ 0<\gamma\le 1\,, \end{equation}
 and in presence of a positive datum  $f$; in Section \ref{strongly} we will treat the critical  case $\theta=1$, while Section \ref{nonnegative} will be devoted to the general case of a nonnegative datum $f$ and $\gamma>1$.

\medskip

We start clarifying the notion of solution to \eqref{pb} in this case: 
\begin{defin}
	\label{weakdefpositive}
	A nonnegative function $u\in BV(\Omega)$ is a solution to problem \eqref{pb} if $g(u^*)\in L^1_{\rm loc}(\Omega, |Du|)$, $h(u)f \in L^1_{\rm loc}(\Omega)$ and if there exists $z\in \mathcal{D}\mathcal{M}^\infty(\Omega)$ with $||z||_{L^\infty(\Omega)^N}\le 1$ such that
	\begin{align}
	&-\operatorname{div}z = g(u^*)|Du| + h(u)f \ \ \text{as measures in }\Omega, \label{def_distrp=1}
	\\
	&(z,Du)=|Du| \label{def_zp=1} \ \ \ \ \text{as measures in } \Omega,
	\\
	&u(x)(1 + [z,\nu] (x))=0 \label{def_bordop=1}\ \ \ \text{for  $\mathcal{H}^{N-1}$-a.e. } x \in \partial\Omega.
	\end{align}
\end{defin}
\begin{remark}\label{rempos}
	Let us underline some relevant facts about the above Definition \ref{weakdefpositive}. First of all the use of  $g(u^*)$ is technically needed in order to give sense a priori  to \eqref{def_distrp=1}. As a matter of fact,  since we will show that in any cases  $u$ does not possess jump part then $g(u^*)$ can be regarded as $g(u)$ (actually $g(\tilde u)$, where $\tilde u $ its Lebesgue representative) once integrated against a measure that is absolutely continuous with respect to $\mathcal{H}^{N-1}$. 
	
Also observe that, if $h(0)=\infty$, then  $h(u)f\in L^1_{\rm loc}(\Omega)$ implies that $\{u=0\}\subset \{f=0\}$; in particular, as here $f>0$,   this  means that $u>0$ a.e. in $\Omega$. We also highlight that, as nowadays classical since \cite{ABCM},  \eqref{def_zp=1} is the way $z$ is intended to represent the quotient $|Du|^{-1}Du$. Finally condition \eqref{def_bordop=1} is the weak sense in which the Dirichlet datum is meant; it roughly asserts  that either $u$ has zero trace or the weak trace of the normal component of $z$ has least possible slope at the boundary.
\end{remark}
We are ready to state the existence result in this mild singular case: 
\begin{theorem}\label{teomain}
	Let $0<f\in L^N(\Omega)$ such that $||f||_{L^N(\Omega)}\mathcal{S}_1h(\infty)<1$ and let $g$ and $h$ satisfy  resp. \eqref{g1} and \eqref{h1} with \eqref{mild} in force. Then there exists a solution to problem \eqref{pb} in the sense of Definition \ref{weakdefpositive}.
\end{theorem}

\begin{remark}\label{g=0}
The assumption on the datum $f$ can be slightly relaxed; in fact, reasoning as in \cite[Section 7]{DGOP} one can easily cover the case of $f$ belonging to the Lorentz space $L^{N,\infty}(\Omega)$ by substituting   the smallness assumption on the product $||f||_{L^N(\Omega)}{\mathcal{S}}_1 h(\infty)$ with 
\begin{equation} \label{lore}
||f||_{L^N(\Omega)}\tilde{\mathcal{S}}_1 h(\infty) <1\,,
\end{equation}
where $\tilde{\mathcal S}_{1}= [(N-1)\omega_{N}^{\frac{1}{N}}]^{-1}$ is the best constant in the embedding of $W^{1,1}_0 (\Omega)$ into $L^{\frac{N}{N-1},1}(\Omega)$ (where, as usual,  $\omega_N$ indicates the volume of the ball of radius 1 in $\rn$). 

Assumption \eqref{lore}   reveals a critical threshold  as it can be deduced by comparison  with the case $g\equiv0$  (see for instance \cite{CT,MST1,DGOP}); in this sense \eqref{lore} is sharp. 
\end{remark}

We also have the following regularity result on the solution given by Theorem \ref{teomain} whose proof will  follow by Lemma \ref{corou} and Lemma \ref{lemmasalto} below.  

\begin{theorem}\label{teoreg}
Under the same assumptions  the solution found  in  Theorem \ref{teomain} satisfies that $u\in L^\infty(\Omega)$ and $D^j u = 0$. 
\end{theorem}

\begin{remark}\label{3.6}
As will be clear by the proof of Lemma \ref{lemmasalto} below, the fact that solutions of problem \eqref{pb} do not possess jump parts is a regularizing effect given by the presence of the gradient term $g(u)|D u|$; a similar situation was noticed in \cite{mazonsegura, ds, ls} while in the case $g(s)\equiv 0$ solutions can have a nontrivial jump part (see \cite{dgs, DGOP}). Also the fact that $u$ is bounded is quite natural and this is essentially due to the presence of the zero order term $h(u)f$.  In fact, as we will see, the solution we found lies underneath the solution of problem \eqref{pb} with  $g(s)\equiv 0$ found in \cite[Theorem 3.3]{DGOP}, that is shown to be bounded.  In other words,  the perturbation given by the gradient term does not make the situation worse with respect to  boundedness of the solution a rough reason being the following:  as it will be hinted by explicit examples (see Example \ref{example1} and Example \ref{example2} below)   the solutions tend to be  nearly "constant" inside $\Omega$ so that   $g(u)|D u|$ only acts near the boundary where the solution would like to be small (though not necessarily zero).  
\end{remark}

At  certain points we shall make use of  the following auxiliary function 
\begin{equation*}\label{gamma}
\Gamma(s)=  \displaystyle 
\int_{0}^{s} g(t) \ dt.
\end{equation*}
We explicitly  observe that, as $g$ is assumed to be positive and $0<\theta<1$, one has that $\Gamma (s)$ is well defined in $[0,\infty)$, and that there exists $\Gamma^{-1}(s)$ which is locally Lipschitz in $[0,\infty)$.

\medskip\medskip 

The proof of both Theorem \ref{teomain} and \ref{teoreg} will be built  in few steps and it will	 be completed in Section \ref{3.4} once we have all the ingredients at our disposal. In Section \ref{3.1} we  introduce the approximation scheme that will lead us to the proof of Theorem \ref{teomain} and we describe the main strategy that will involve the use of the associated auxiliary problem with no gradient term.   Section \ref{3.2} is devoted to obtain the basics a priori estimates on the approximating problems and to the identification of a bounded limit function $u$.  In Section \ref{3.3} we check a crucial property of this limit function, that is $u$ possesses no jump part. Finally, Section \ref{3.4} will be dedicated to the passage to the limit and consequently to the proof of our existence and regularity  results.

\subsection{An auxiliary problem and the approximation scheme}
\label{3.1}

In order to better introduce our strategy of the proof of Theorem \ref{teomain} it is worth   recalling  some basic facts about the case $g\equiv 0$. In  \cite{DGOP}, it is proved that  there exists a bounded solution $v$ (that, by the way,  is unique if $f>0$ a.e. and $h$ is decreasing) to 
\begin{equation}
\label{pbzero}
\begin{cases}
\dis -\Delta_1 v  = h(v)f & \text{in}\;\Omega,\\
v=0 & \text{on}\;\partial\Omega,
\end{cases}
\end{equation}
where $f\in L^N(\Omega)$ is nonnegative and $h$ is as above. The notion of solution to problem \eqref{pbzero} is given as in Definition \ref{weakdefpositive}  thought of as $g\equiv 0$. The solution $v$ is constructed through the following approximation scheme:
\begin{equation}
\label{pbpdgop}
\begin{cases}
\dis -\Delta_p v_p  = h_p(v_p)f & \text{in}\;\Omega,\\
v_p=0 & \text{on}\;\partial\Omega,
\end{cases}
\end{equation}
where $h_p(s):= T_{\frac{1}{p-1}}\left(h(s)\right)$. In \cite{DGOP} it is proved that $v_p$ converges almost everywhere to a bounded solution $v$ of \eqref{pbzero}.

\medskip

For our purposes it is also worth stating and sketching the proof of the following result; here we set the  useful  notation 
\begin{equation*}\label{sup}
\begin{aligned}
	h_k(\infty):= \sup_{s\in[k,\infty)} h(s).
\end{aligned}
\end{equation*}
\begin{lemma}\label{lemmalim}
	Let $0\leq f\in L^N(\Omega)$ such that $||f||_{L^N(\Omega)}\mathcal{S}_1h(\infty)<1$ and  let $h$ satisfy \eqref{h1}. Let $v_p$ be a solution to \eqref{pbpdgop} then there exists a positive constant $\tilde{c}$ such that
	 \begin{equation}\label{lim}
	 v\le \tilde{c},
	 \end{equation} 
	 where $v$ is the almost everywhere limit of $v_p$ as $p\to 1^+$ and solves \eqref{pbzero}.
\end{lemma}
\begin{proof}
Let us take $G_k(v_p)$ as a test function in \eqref{pbpdgop} deducing
\begin{equation}\label{stimaulim1dgop}
\int_\Omega|\nabla G_{k}(v_p)|^p = \int_\Omega h_p(v_p)f G_{k}(v_p) \le h_{k}(\infty) ||f||_{L^N(\Omega)} ||G_{k}(v_p)||_{L^{\frac{N}{N-1}}(\Omega)}.
\end{equation}
Now we apply the Young and the Sobolev inequalities on the left hand side of \eqref{stimaulim1dgop} obtaining
\begin{equation}\label{stimaulim2dgop}
\begin{aligned}
\mathcal{S}_1^{-1} ||G_{k}(v_p)||_{L^{\frac{N}{N-1}}(\Omega)} &\le \int_{\Omega}|\nabla G_{k}(v_p)| \le \int_{\Omega}|\nabla G_{k}(v_p)|^p + \frac{p-1}{p}|\Omega| 
\\
&\le   h_{k}(\infty) ||f||_{L^N(\Omega)} ||G_{k}(v_p)||_{L^{\frac{N}{N-1}}(\Omega)} + \frac{p-1}{p}|\Omega|.
\end{aligned}
\end{equation} 
Under the assumption that $||f||_{L^N(\Omega)} \mathcal{S}_1 h(\infty) <1$ then we can pick a  $\tilde{c}$ such that $\mathcal{S}_1^{-1}-||f||_{L^N(\Omega)} h_{\tilde{c}}(\infty) >C>0$ for a constant $C$ which is independent of $p$. Then, choosing $k=\tilde{c}$ in \eqref{stimaulim2dgop}, one gets 
\begin{equation*}
C ||G_{\tilde{c}}(v_p)||_{L^{\frac{N}{N-1}}(\Omega)} \le  \frac{p-1}{p}|\Omega|,
\end{equation*} 
whence, taking $p\to1^+$, by weak lower semicontinuity of the norm, one has $||G_{\tilde{c}}(v)||_{L^{\frac{N}{N-1}}(\Omega)}=0$, that is   $v\le \tilde{c}$
almost everywhere in $\Omega$. Finally, reasoning  as in \cite{DGOP}, one can show that $v$ is a solution to \eqref{pbzero}.
\end{proof}
\begin{remark} \label{czero} We  stress that,   if $h(0)<\infty$, then, as shown in \cite[Theorem 3.4]{DGOP},  $\tilde{c}\equiv 0$  provided the $L^N (\Omega)$ norm of $f$ is small enough; we shall come back on this fact later on (see Remark \ref{5.5} below). 
\end{remark}

Now let us come back to problem \eqref{pb};   following the heuristics given in Remark \ref{3.6},    our strategy will be based on finding a solution to \eqref{pb} which still lives in the interval $[0, \tilde{c}]$ where $v$ does.

First of all  we are interested in deducing some {a priori} estimates for the solutions to the following approximation problem 
\begin{equation}
\label{pbp}
\begin{cases}
\dis -\Delta_p u_p  = g_p(u_p)|\nabla u_p|^p + h_p(u_p)f & \text{in}\;\Omega,\\
u_p=0 & \text{on}\;\partial\Omega,
\end{cases}
\end{equation}
where
$$
g_p(s):= 
\begin{cases}
	T_{\frac{1}{p-1}}\left(g(s)\right) \ \ \ \ &s< \tilde{c},\\
	\displaystyle -\frac{g(\tilde{c})}{p-1}(s-\tilde{c}-p+1)  &\tilde{c}\le s \le \tilde{c} + p-1,\\
	0 &s> \tilde{c}+p-1,
\end{cases}
$$
and without loosing generality, from here on we assume that  $1<p<2$ and that $T_{\frac{1}{p-1}}\left(g(\tilde{c})\right) =g(\tilde{c})$. 

 \begin{figure}[htbp]\centering
\includegraphics[width=2in]{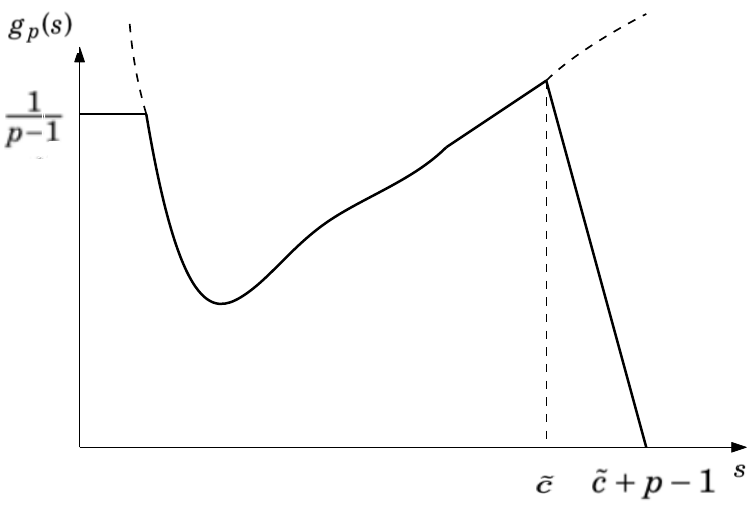}
\caption{$g_p (s)$}
\end{figure}

Moreover 
$h_p(s):= T_{\frac{1}{p-1}}\left(h(s)\right)$; again, without loss of generality, up to choose $p$ near to $1^+$, we can assume that we are (possibly) truncating $h$ only near $0$.

For either $q=1$ or $q=p$, let us also define the auxiliary function  
\begin{equation*}\label{gammap}
\Gamma_{p,q}(s) := \displaystyle 
\int_{0}^{s} g_p^{\frac{1}{q}}(t) \ dt\,;
\end{equation*} 
observe that, in both cases, $\Gamma_{p,q}(s)$ converges to $\Gamma(s)$ for any $s\ge0$ as $p\to 1^+$.

The existence of a nonnegative solution $u_p \in W^{1,p}_0(\Omega)\cap L^\infty(\Omega)$ of problem \eqref{pbp} is granted by the following argument: in \cite{ps}, it is proved the existence of a solution to 
\begin{equation*}
\begin{cases}
\dis -\Delta_p w  = g_p(w)|\nabla w|^p + h_p(v)f & \text{in}\;\Omega,\\
w=0 & \text{on}\;\partial\Omega,
\end{cases}
\end{equation*}
for any  nonnegative $v$ belonging to $L^p(\Omega)$; as the result comes along with standard associated  estimates,  it is easy to check that the  application $T: L^p(\Omega)\mapsto L^p(\Omega)$ such that $T(v)=w$ admits a fixed point, which is a solution to \eqref{pbp}.

\subsection{A priori estimates and boundedness of the limit function}\label{3.2}

As far as  the estimates are concerned the positivity of $f$  it is not relevant, so that  in this section we will consider the more general case of a nonnegative datum $f\in L^N(\Omega)$; this generalization will be useful in Section \ref{nonnegative}.  As it is clear the main issues shall  rely on proving that solutions to \eqref{pbp} enjoy some estimates   which are independent of $p$ (at least for $p\sim 1^+$). The first result contains an estimate on a suitable exponential of $u_p$ and it also make use of the following elementary inequality. 

\begin{lemma}\label{ineq}
For a fixed $C>1$, there exists $p_0 \in (1,2)$ such that  
	\begin{equation*}
	e^{2s}-1\le C \left(e^{\frac{2s}{p}}-1\right)^p, \ \ \ \forall s\ge 1,
	\end{equation*}
	holds for any $p\in (1,p_0)$. 
\end{lemma}

An easy consequence of Lemma \ref{ineq} is that for a sequence of constant $C_p$  such that $C_p\to 1^+$ as $p\to 1^+$ one has 
	\begin{equation}\label{sequencep}
	e^{2s}-1\le C_p \left(e^{\frac{2s}{p}}-1\right)^p, \ \ \ \forall s\ge 1\,,
	\end{equation}
	up to suitably choosing a sequence of  $p$'s converging to $1^+$. 
\begin{lemma}\label{lemmastima}
Let $g$,  $h$ satisfy  \eqref{g1} and \eqref{h1} with \eqref{mild} in force,  and let $0\le f\in L^N(\Omega)$ such that $||f||_{L^N(\Omega)}\mathcal{S}_1h(\infty)< 1$. Let $u_p$ be a solution of \eqref{pbp} then there exists $p_0>1$ such that for any $\eta>0$ 
\begin{equation}\label{stimap}
\int_{\Omega} |\nabla (e^{\eta u_p}-1)|^p \le C,
\end{equation}
for some constant $C$ which does not depend on $p\in (1,p_0)$;   moreover, 
\begin{equation}\label{stimabv}
|| e^{\eta u_p}-1||_{W^{1,1}_0(\Omega)}\leq C,
\end{equation}
again for some $C$ not depending on  $p\in (1,p_0)$. 
\end{lemma}
\begin{proof}
	We start taking  $e^{\eta pG_{k}(u_p)}-1$ with $\eta\ge 1$ and $k\ge\tilde{c} + 1$ as a test function (we recall $1<p<2$ and that $g_p(s)\equiv 0$ if $s\ge \tilde{c}+p-1$) in the weak formulation of \eqref{pbp}. Hence one has that 
	\begin{equation*}
	\begin{aligned}
	\eta p\int_\Omega |\nabla G_k(u_p)|^p e^{\eta pG_k(u_p)} &= \int_\Omega h_p(u_p)f(e^{\eta pG_k(u_p)}-1) \le h_k(\infty)\int_\Omega f (e^{\eta pG_k(u_p)}-1)
	\\
	&= h_k(\infty)\int_{\{k\le u_p\le k+2\}} f  (e^{\eta pG_k(u_p)}-1) + h_k(\infty)\int_{\{u_p > k+2\}}f  (e^{\eta pG_k(u_p)}-1)
	\\
	&\le h_k(\infty)e^{2\eta p}||f||_{L^1(\Omega)} + C_p h_k(\infty)\int_{\{u_p > k+2\}}f (e^{\eta G_k(u_p)}-1)^p,
	\end{aligned}
	\end{equation*}
	in which we have also used \eqref{sequencep}.	At this point one can apply the H\"older and the Sobolev inequalities deducing 
	\begin{equation*}\label{stimagrad}
	\begin{aligned}
	\eta p\int_\Omega |\nabla G_k(u_p)|^p e^{\eta pG_k(u_p)} &\le h_k(\infty) e^{2\eta p} ||f||_{L^1(\Omega)} + C_p h_k(\infty)||f||_{L^N(\Omega)} \left(\int_{\Omega}(e^{\eta G_k(u_p)}-1)^{\frac{pN}{N-p}}\right)^{\frac{N-p}{N}}|\Omega|^{\frac{p-1}{N}}
	\\
	&
	\le h_k(\infty) e^{2\eta p} ||f||_{L^1(\Omega)} + C_p h_k(\infty)||f||_{L^N(\Omega)} \mathcal{S}_p^p \eta^p\int_{\Omega}|\nabla G_k(u_p)|^p e^{\eta pG_k(u_p)}|\Omega|^{\frac{p-1}{N}},
	\end{aligned}
	\end{equation*}
	and one can observe that for any $\eta \ge 1$ there exists $\tilde{k}>0$ sufficiently large and $p_0$ sufficiently close to $1$ such that
	$$\eta p- C_p h_k(\infty)||f||_{L^N(\Omega)} \mathcal{S}_p^p \eta^p |\Omega|^{\frac{p-1}{N}}>C>0, \ \ \ \forall k\ge \tilde{k}, \ \forall p\in (1, p_0),$$
	for a constant $C$ not depending on both $p$, $k$.
	Hence one has that for any $\eta>0$ (by monotonicity)
	\begin{equation}\label{stimagk2}
	\int_\Omega |\nabla G_{k}(u_p)|^p e^{\eta pG_{k}(u_p)} \le C,
	\end{equation}
	for some positive constant $C$ independent of $p\in(1,p_0)$ and for every $k\ge \tilde{k}\ge \tilde{c} + 1$. Now we decompose $u_p$ as
	$$u_p= T_{\tilde{k}}(u_p) + G_{\tilde{k}}(u_p) = T_{\varepsilon}(u_p) + T_{\tilde{k}-\varepsilon}(G_{\varepsilon}(u_p)) + G_{\tilde{k}}(u_p),$$
	 where $\varepsilon$ will be fixed small enough. 
We take $T_\varepsilon(u_p)$ as a test function in the weak formulation solved by \eqref{pbp} yielding to
\begin{equation*}\label{stimatk}
\begin{aligned}
&\int_{\Omega} |\nabla T_\varepsilon (u_p)|^p \le \left(c_1\varepsilon^{1-\theta}+\sup_{s\in[s_1,\tilde{c}+1)}g(s)\varepsilon\right)\int_\Omega |\nabla u_p|^p + \left(c_2\varepsilon^{1-\gamma} + \sup_{s\in[s_2,\infty)}h(s)\varepsilon\right) \int_\Omega f
\\
&=\left(c_1\varepsilon^{1-\theta}+\sup_{s\in[s_1,\tilde{c}+1)}g(s)\varepsilon\right)\int_\Omega |\nabla T_\varepsilon(u_p)|^p + \left(c_1\varepsilon^{1-\theta}+\sup_{s\in[s_1,\tilde{c}+1)}g(s)\varepsilon\right)\int_\Omega |\nabla T_{\tilde{k}-\varepsilon}(G_\varepsilon(u_p))|^p
\\
& + \left(c_2\varepsilon^{1-\gamma} + \sup_{s\in[s_2,\infty)}h(s)\varepsilon\right) \int_\Omega f,
\end{aligned}	
\end{equation*}
whence, fixing $\varepsilon$ such that $1-\left(c_1\varepsilon^{1-\theta}+\displaystyle \sup_{s\in[s_1,\tilde{c}+1)}g(s)\varepsilon\right)>C^{-1}$ for some positive $C$, then one is lead to 
\begin{equation}\label{stima1}
\begin{aligned}
\int_{\Omega} |\nabla T_{\varepsilon}(u_p)|^p &\le C\left(c_1\varepsilon^{1-\theta}+\sup_{s\in[s_1,\tilde{c}+1)}g(s)\varepsilon\right)\int_\Omega |\nabla T_{\tilde{k}-\varepsilon}(G_{\varepsilon}(u_p))|^p
\\
&+C\left(c_2\varepsilon^{1-\gamma} + \sup_{s\in[s_2,\infty)}h(s)\varepsilon\right)\int_\Omega f.
\end{aligned}	
\end{equation}	
We just need to estimate the first term on the right hand side of the previous inequality.
Now we test the weak formulation of \eqref{pbp} with $e^{\eta pT_{\tilde{k}-\varepsilon}(G_{\varepsilon}(u_p))}-1$ obtaining
\begin{equation*}
\begin{aligned}
& \eta p\int_\Omega |\nabla T_{\tilde{k}-\varepsilon}(G_{\varepsilon}(u_p))|^p e^{\eta pT_{\tilde{k}-\varepsilon}(G_{\varepsilon}(u_p))} = \int_{\Omega} g_p(u_p)|\nabla u_p|^p(e^{\eta pT_{\tilde{k}-\varepsilon}(G_{\varepsilon}(u_p))}-1) \\ & + \int_\Omega h_p(u_p)f(e^{\eta pT_{\tilde{k}-\varepsilon}(G_{\varepsilon}(u_p))}-1)
\le \sup_{s\in[\varepsilon,\tilde{c}+1)}g(s) \int_{\Omega} |\nabla T_{\tilde{k}-\varepsilon}(G_{\varepsilon}(u_p))|^pe^{\eta pT_{\tilde{k}-\varepsilon}(G_{\varepsilon}(u_p))} \\ &  + e^{\eta p(\tilde{k}-\varepsilon)}\sup_{s\in[\varepsilon,\infty)}h(s)\int_\Omega f.
\end{aligned}
\end{equation*}
Moreover there exists $\eta_1$ such that $\eta p - \displaystyle \sup_{s\in[\varepsilon,\tilde{c}+1)}g(s)>C>0$ for every $\eta\ge \eta_1$ and where $C$ does not depend on $p$. Hence one has
\begin{equation}\label{stimatkgk}
\begin{aligned}
&C \int_\Omega |\nabla T_{\tilde{k}-\varepsilon}(G_{\varepsilon}(u_p))|^p e^{\eta pT_{\tilde{k}-\varepsilon}(G_{\varepsilon}(u_p))} \le e^{\eta p(\tilde{k}-\varepsilon)}\sup_{s\in[\varepsilon,\infty)}h(s)\int_\Omega f \le C.
\end{aligned}
\end{equation} 
Therefore, it follows by using \eqref{stimatkgk} in \eqref{stima1} and by monotonicity, that 
$$\int_\Omega |\nabla T_{\varepsilon}(u_p)|^p e^{\eta pT_{\varepsilon}(u_p)} \le e^{\eta p\varepsilon} \int_\Omega |\nabla T_{\varepsilon}(u_p)|^p  \le  C,$$
which, gathered with \eqref{stimagk2} and \eqref{stimatkgk}, concludes the proof of \eqref{stimap}.

In order to check \eqref{stimabv} we use  H\"older inequality  and \eqref{stimap} one has 
	\begin{equation*}
	\begin{aligned}\label{bv}
	\int_{\Omega} |\nabla (e^{\eta u_p}-1)| \le \eta^p\int_{\Omega} |\nabla u_p|^pe^{\eta p u_p}  + \frac{p-1}{p} |\Omega|\le C,		   
	\end{aligned}	
	\end{equation*}			
	where $C$ is independent of $p\in (1,p_0)$. 

\end{proof}

Now we refine the estimates of Lemma \ref{lemmastima}, deducing  that  the sequence $u_p$ is bounded  in $BV(\Omega)$. This will take to the existence of a limit function which will be the candidate to be a solution for \eqref{pb}. Here we also show that this limit function is less or equal than $\tilde{c}$ almost everywhere, which is the value given by Lemma \ref{lemmalim}.
\begin{lemma}\label{corou}
		Under the assumptions of Lemma \ref{lemmastima} there exists $p_0>1$ such that $u_p$ is bounded in $BV(\Omega)$  and $\Gamma_{p,p}(u_p)$ is bounded in $BV_{\rm loc}(\Omega)$ uniformly with respect to $p\in (1,p_0)$. Moreover there exists $u\in BV(\Omega)\cap L^\infty(\Omega)$ such that $u_p$ converges to $u$ (up to a subsequence) in $L^q(\Omega)$ for every  $q<\infty$ \bk and $\nabla u_p$ converges to $D u$ locally *-weakly as measures. Finally it holds
		\begin{equation}\label{bounded}
		||u||_{L^\infty(\Omega)}\le \tilde{c}.
		\end{equation}
\end{lemma}	
\begin{proof}
	As an easy consequence of \eqref{stimabv}, $u_p$ is also bounded in $BV(\Omega)$.  Hence a standard compactness argument allows to deduce that there exists $u\in BV(\Omega)$ such that, up to subsequences, $ u_p$ converges to $u$ in $L^q(\Omega)$ for every $q<\frac{N}{N-1}$. Moreover $\nabla u_p$ converges to $D u$ locally *-weakly as measures as $p\to 1^+$. Observe that  \eqref{stimabv} also implies the strong convergence of  $e^{\eta u_p}-1$  in $L^1(\Omega)$ and so it follows that $u_p$ converges to $u$   in $L^q(\Omega)$ for any $q<\infty$ as $p\to 1^+$.

	Now we take a nonnegative $\varphi\in C^1_c(\Omega)$ as a test function in the distributional formulation of \eqref{pbp}, we get rid of the term involving $h$ and we apply the chain rule and the Young inequality, yielding to
	\begin{equation*}
	\int_\Omega |\nabla \Gamma_{p,p}(u_p)|\varphi  \le \int_\Omega |\nabla u_p|^{p-1}|\nabla \varphi| +\frac{p-1}{p}\int_\Omega \varphi \le \int_\Omega |\nabla u_p|^{p} + \int_\Omega |\nabla \varphi|^p + \frac{p-1}{p}\int_\Omega \varphi \le  C,
	\end{equation*}
	which, jointly with the fact that $\Gamma_{p,p}(u_p)$ is bounded in $L^\infty(\Omega)$,
	 implies that $\Gamma_{p,p}(u_p)$ is bounded in $BV_{\rm loc}(\Omega)$ with respect to $p$.	
	\\ It remains to show that \eqref{bounded} holds. Let $k_p= \tilde{c}+p-1$ and let take $G_{k_p}(u_p)$ as a test function in \eqref{pbp} obtaining
	\begin{equation}\label{stimaulim1}
	\int_\Omega|\nabla G_{k_p}(u_p)|^p = \int_\Omega h_p(u_p)f G_{k_p}(u_p) \le h_{k_p}(\infty) ||f||_{L^N(\Omega)} ||G_{k_p}(u_p)||_{L^{\frac{N}{N-1}}(\Omega)}.
	\end{equation}
	Now we apply the Young and the Sobolev inequalities on the left hand side of \eqref{stimaulim1} yielding to 
	\begin{equation*}\label{stimaulim2}
	\mathcal{S}_1^{-1} ||G_{k_p}(u_p)||_{L^{\frac{N}{N-1}}(\Omega)} \le h_{k_p}(\infty) ||f||_{L^N(\Omega)} ||G_{k_p}(u_p)||_{L^{\frac{N}{N-1}}(\Omega)} + \frac{p-1}{p}|\Omega|.
	\end{equation*} 
	We recall that $\tilde{c}$ is such that  $\mathcal{S}_{1}^{-1} -||f||_{L^N(\Omega)} h_{\tilde{c}}(\infty)>C>0$ for a constant $C$ which is independent of $p$ near $1$.
	Now observe that  $\mathcal{S}_{1}^{-1} -||f||_{L^N(\Omega)} h_{k_p}(\infty) > \mathcal{S}_{1}^{-1}  -||f||_{L^N(\Omega)} h_{\tilde{c}}(\infty) $ and then 
	\begin{equation*}
	C ||G_{k_p}(u_p)||_{L^{\frac{N}{N-1}}(\Omega)} \le  \frac{p-1}{p}|\Omega|,
	\end{equation*} 
whence, taking $p\to1^+$ and by weak lower semicontinuity of the norm, one has $||G_{\tilde{c}}(u)||_{L^{\frac{N}{N-1}}(\Omega)}=0$, namely $u\le \tilde{c}$ almost everywhere in $\Omega$.  The proof is concluded.
\end{proof}

\subsection{The limit $u$ has no jumps}\label{3.3}

We now prove the following  result  in which we  show that $D u$ has no jump part.

\begin{lemma}\label{lemmasalto}
	Let $0\le u\in BV_{\rm loc}(\Omega)\cap L^{\infty}(\Omega)$, $z\in\mathcal{DM}^{\infty}_{\rm loc}(\Omega)$ with $\|z\|_{L^{\infty}(\Omega)^N}\le 1$, and let $\Gamma:[0,\infty)\to[-\infty,+\infty)$ be a continuous {strictly} increasing function, finite on $(0,\infty)$, such that $\Gamma(u)\in BV_{\rm loc}(\Omega)$. Let $\tilde{f},\tilde{h}\in L^{1}_{\rm loc}(\Omega)$ and assume that
	\[
	-\operatorname{div} z \ge |D\Gamma(u)| + \tilde{f} \qquad \text{as measures in } \Omega,
	\]
	and that one of the following holds:
	\begin{itemize}
		\item[(i)] $	-\operatorname{div}\!\left(e^{\Gamma(u)}z\right) \le \tilde{h} \quad \text{as measures in } \Omega;$
		\item[(ii)] $\Gamma(0)=0$, $\chi_{\{u>0\}}\in BV_{\rm loc}(\Omega)$, and $-\operatorname{div}\!\left(e^{\Gamma(u)}\chi_{\{u>0\}}z\right) + |D\chi_{\{u>0\}}| \le \tilde{h} \quad \text{as measures in } \Omega$.
	\end{itemize}
	Then $D^{j}u=0$.
\end{lemma}

\begin{proof}
First of all we prove that $D^{j}\Gamma(u)=0$. Indeed, since $\Gamma(u)\in BV_{\rm loc}(\Omega)$,
by Theorem 3.78 of \cite{afp}, $S_{\Gamma(u)}$ is (locally) countably $\mathcal{H}^{N-1}$-rectifiable
and then there exist $C^{1}$ hypersurfaces $\xi_{k}$ such that
\[
\mathcal{H}^{N-1}\!\left(S_{\Gamma(u)}\setminus\bigcup_{k=1}^{\infty}\xi_{k}\right)=0.
\]
Hence we will just need to show that, for any $k\in\mathbb{N}$, $|D^{j}\Gamma(u)|(\xi_{k})=0$.
In particular, the proof will be done once one shows that for any $x_{0}\in\xi_{k}$ there exists
an open neighbourhood $U$ of $x_{0}$ with $U\cap\xi_{k}\subset\subset\Omega$ such that
$|D^{j}\Gamma(u)|(U\cap\xi_{k})=0$. In order to prove it, let $\tilde{\nu}$ be a continuous
unit normal orienting $\xi_{k}$; since $\xi_{k}$ is of class $C^{1}$, we can choose an open
neighbourhood $U$ of $x_{0}$ with $U\cap\xi_{k}\subset\subset\Omega$ such that $U\cap\xi_{k}$
is the graph of a $C^{1}$ function (up to the boundary) over an open $(N-1)$-dimensional ball
in the hyperplane orthogonal to $\tilde{\nu}(x_{0})$,
and consider the open cylinder
\[
U_{n}:=\left\{x+t\tilde{\nu}(x_{0}):x\in U\cap\xi_{k},\,|t|<\frac{1}{n}\right\},\qquad n\ge n_{0},
\]
where $n_{0}$ is fixed such that $d(\overline{U\cap\xi_{k}},\partial\Omega)>\frac{1}{n_{0}}$
($d$ is the usual distance function). Then each $U_{n}$ is a Lipschitz open set compactly
contained in $\Omega$ satisfying $\bigcap_{n\ge n_{0}}U_{n}=U\cap\xi_{k}$. We decompose
$\partial U_{n}=E_{n}^{+}\cup E_{n}^{-}\cup E_{n}^{0}$, where
$E_{n}^{\pm}:=\{x\pm\frac{1}{n}\tilde{\nu}(x_{0}):x\in U\cap\xi_{k}\}$ and $E_{n}^{0}$ is the
lateral surface of $U_{n}$, noting that $E_{n}^{0}\subset\partial U_{n_{0}}$ for every
$n\ge n_{0}$.

Let us fix a nonnegative $\varphi\in C_{c}^{1}(U_{n_{0}})$, so that $\varphi=0$ on $E_{n}^{0}$
for all $n\ge n_{0}$. For any $F\in\mathcal{DM}^{\infty}_{\rm loc}(\Omega)$, denoting by $\nu$
the outward unit normal to $\partial U_{n}$, the Green formula (2.3) on $U_{n}$ yields
\[
\int_{U_{n}}\varphi\,\mathrm{div}F+\int_{U_{n}}F\cdot\nabla\varphi
=\int_{\partial U_{n}}\varphi[F,\nu]\,d\mathcal{H}^{N-1}.
\]
Reasoning as in Lemma 4 of \cite{ads}, and taking into account that a vector field
$F\in\mathcal{DM}^{\infty}_{\rm loc}(\Omega)$ admits two (possibly distinct) one-sided normal
traces $[F,\tilde{\nu}]^{+}$ and $[F,\tilde{\nu}]^{-}$ on $U\cap\xi_{k}$ with respect to the
orientation $\tilde{\nu}$ {(see \cite{acm} and \cite[Section 2.4]{cdc}), which satisfy
$\|[F,\tilde{\nu}]^{\pm}\|_{L^{\infty}(U\cap\xi_{k})}\le\|F\|_{L^{\infty}(\Omega)^N}$}, one has
\[
\lim_{n\to\infty}\int_{E_{n}^{+}}\varphi[F,\nu]\,d\mathcal{H}^{N-1}
=\int_{U\cap\xi_{k}}\varphi[F,\tilde{\nu}]^{+}\,d\mathcal{H}^{N-1},
\qquad
\lim_{n\to\infty}\int_{E_{n}^{-}}\varphi[F,\nu]\,d\mathcal{H}^{N-1}
=-\int_{U\cap\xi_{k}}\varphi[F,\tilde{\nu}]^{-}\,d\mathcal{H}^{N-1},
\]
while $\int_{E_{n}^{0}}\varphi[F,\nu]\,d\mathcal{H}^{N-1}=0$. Passing to the limit as
$n\to\infty$ in the Green formula above (recalling that $\bigcap_{n\ge n_{0}}U_{n}=U\cap\xi_{k}$
and $|U\cap\xi_{k}|=0$), we obtain {(see also \cite[Section 2.4]{cdc})}
\[
-\int_{U\cap\xi_{k}}\varphi \operatorname{div}  F
=\int_{U\cap\xi_{k}}\varphi\big([F,\tilde{\nu}]^{-}-[F,\tilde{\nu}]^{+}\big)\,d\mathcal{H}^{N-1}.
\]

We first use this identity with $F=z$ in the inequality $-\mathrm{div}z\ge|D\Gamma(u)|+\tilde{f}$.
Integrating against $\varphi$ over $U_{n}$ and letting $n\to\infty$, since
$\tilde{f}\in L^{1}_{\rm loc}(\Omega)$ and $|U\cap\xi_{k}|=0$, the term involving $\tilde{f}$
vanishes, while by Proposition 3.92 of \cite{afp} one has
$|D\Gamma(u)|\llcorner(U\cap\xi_{k})=|\Gamma(u)^{+}-\Gamma(u)^{-}|\mathcal{H}^{N-1}\llcorner(U\cap\xi_{k})$,
where $\Gamma(u)^{\pm}$ denote the (finite $\mathcal{H}^{N-1}$-a.e.) one-sided traces of
$\Gamma(u)$ on $U\cap\xi_{k}$ with respect to $\tilde{\nu}$. By the arbitrariness of
$0\le\varphi\in C_{c}^{1}(U_{n_{0}})$, it follows that
\begin{equation}\label{eq:jump1}
	[z,\tilde{\nu}]^{-}-[z,\tilde{\nu}]^{+}\ge|\Gamma(u)^{+}-\Gamma(u)^{-}|
	\qquad\mathcal{H}^{N-1}\text{-a.e. on }U\cap\xi_{k}.
\end{equation}
Next, set $w\equiv 1$ in case (i) and $w:=\chi_{\{u>0\}}$ in case (ii). Since
$\Gamma(u)\le M:=\Gamma(\|u\|_{L^{\infty}(\Omega)})<\infty$ a.e. in $\Omega$ and the function
$t\mapsto e^{t}$ is globally Lipschitz on $(-\infty,M]$ (with the convention $e^{-\infty}=0$),
the chain rule in $BV$ implies that
$e^{\Gamma(u)}\in BV_{\rm loc}(\Omega)\cap L^{\infty}(\Omega)$. Hence
$v:=e^{\Gamma(u)}w\in BV_{\rm loc}(\Omega)\cap L^{\infty}(\Omega)$, so that
$vz\in\mathcal{DM}^{\infty}_{\rm loc}(\Omega)$ and $-\mathrm{div}(vz)+|Dw|\le\tilde{h}$ as measures
in $\Omega$ in both cases. Moreover, the one-sided traces of $v$ on $U\cap\xi_{k}$ are given
$\mathcal{H}^{N-1}$-a.e. by $v^{\pm}=e^{\Gamma(u)^{\pm}}w^{\pm}$ (with $w^{\pm}\in\{0,1\}$),
and we also have
$[vz,\tilde{\nu}]^{\pm}=e^{\Gamma(u)^{\pm}}w^{\pm}[z,\tilde{\nu}]^{\pm}$
$\mathcal{H}^{N-1}$-a.e. on $U\cap\xi_{k}$ {(see \cite[Proposition 3.1]{cdc})}. Applying the limit identity above with $F=vz$ to
$-\mathrm{div}(vz)+|Dw|\le\tilde{h}$ (noting that $\lim_{n\to\infty}\int_{U_{n}}\varphi\tilde{h}=0$
and $|Dw|\llcorner(U\cap\xi_{k})=|w^{+}-w^{-}|\mathcal{H}^{N-1}\llcorner(U\cap\xi_{k})$), the
arbitrariness of $0\le\varphi\in C_{c}^{1}(U_{n_{0}})$ yields
\begin{equation}\label{eq:jump2}
	e^{\Gamma(u)^{-}}w^{-}[z,\tilde{\nu}]^{-}-e^{\Gamma(u)^{+}}w^{+}[z,\tilde{\nu}]^{+}+|w^{+}-w^{-}|\le 0
	\qquad\mathcal{H}^{N-1}\text{-a.e. on }U\cap\xi_{k}.
\end{equation}

Now fix a point $x\in U\cap\xi_{k}$ where \eqref{eq:jump1}, \eqref{eq:jump2}, and
$|[z,\tilde{\nu}]^{\pm}(x)|\le 1$ hold (with $\Gamma(u)^{\pm}(x)$ finite), and suppose by
contradiction that $\Gamma(u)^{+}(x)\neq\Gamma(u)^{-}(x)$. Since both \eqref{eq:jump1} and
\eqref{eq:jump2} are invariant under reversing the orientation $\tilde{\nu}\mapsto-\tilde{\nu}$
(which swaps the superscripts $+$ and $-$ and changes the sign of $[z,\tilde{\nu}]^{\pm}$), we
may assume without loss of generality that $\Gamma(u)^{+}(x)>\Gamma(u)^{-}(x)$. For readability,
in the remainder of this argument we drop the explicit dependence on $x$.

We first identify the values of $w^{+}\in\{0,1\}$ and $w^{-}\in\{0,1\}$. In case (i) one has
$w\equiv 1$, so trivially $w^{+}=w^{-}=1$. In case (ii), where $w=\chi_{\{u>0\}}$ and
$\Gamma(0)=0$, the nonnegativity of $u$ and the monotonicity of $\Gamma$ yield
$\Gamma(u)\ge\Gamma(0)=0$ a.e. in $\Omega$, whence $\Gamma(u)^{+}>\Gamma(u)^{-}\ge 0$. Since
$\Gamma(u)$ converges in measure on the $+\tilde{\nu}$ side of $\xi_{k}$ near $x$ to the
strictly positive value $\Gamma(u)^{+}>0$, and since $\{\Gamma(u)>0\}=\{u>0\}$ (as $\Gamma$ is
strictly increasing with $\Gamma(0)=0$), the characteristic function $w=\chi_{\{u>0\}}$
converges in measure to $1$ on that side, which implies $w^{+}=1$. Thus $w^{+}=1$ in both cases,
and we are left with only two possibilities for $w^{-}$.

If $w^{-}=1$ (which is always the case in (i)), then $|w^{+}-w^{-}|=0$ and \eqref{eq:jump2}
becomes
\[
e^{\Gamma(u)^{-}}[z,\tilde{\nu}]^{-}-e^{\Gamma(u)^{+}}[z,\tilde{\nu}]^{+}\le 0,
\qquad\text{i.e.}\qquad
[z,\tilde{\nu}]^{+}\ge e^{-(\Gamma(u)^{+}-\Gamma(u)^{-})}[z,\tilde{\nu}]^{-}.
\]
Substituting this lower bound into \eqref{eq:jump1} and recalling that
$1-e^{-(\Gamma(u)^{+}-\Gamma(u)^{-})}>0$ (as $\Gamma(u)^{+}-\Gamma(u)^{-}>0$) and
$[z,\tilde{\nu}]^{-}\le 1$, we obtain
\begin{align*}
	\Gamma(u)^{+}-\Gamma(u)^{-}
	&\le [z,\tilde{\nu}]^{-}-[z,\tilde{\nu}]^{+} \\
	&\le \big(1-e^{-(\Gamma(u)^{+}-\Gamma(u)^{-})}\big)[z,\tilde{\nu}]^{-}
	\le 1-e^{-(\Gamma(u)^{+}-\Gamma(u)^{-})},
\end{align*}
which contradicts the elementary inequality $1-e^{-s}<s$ for all $s>0$.

If instead $w^{-}=0$ (which can only occur in case (ii)), then $w=\chi_{\{u>0\}}$ converges in
measure to $0$ on the $-\tilde{\nu}$ side of $\xi_{k}$ near $x$. Equivalently, $u=0$ (and hence
$\Gamma(u)=\Gamma(0)=0$) on a set of density $1$ on that side, which forces $\Gamma(u)^{-}=0$
and therefore $\Gamma(u)^{+}>0$. Since $w^{+}=1$ and $w^{-}=0$, we have $|w^{+}-w^{-}|=1$, so
\eqref{eq:jump2} simplifies to
\[
-e^{\Gamma(u)^{+}}[z,\tilde{\nu}]^{+}+1\le 0,
\qquad\text{i.e.}\qquad
[z,\tilde{\nu}]^{+}\ge e^{-\Gamma(u)^{+}}.
\]
Combining this with \eqref{eq:jump1} (where $\Gamma(u)^{-}=0$) and $[z,\tilde{\nu}]^{-}\le 1$,
we get
\[
\Gamma(u)^{+}=\Gamma(u)^{+}-\Gamma(u)^{-}
\le [z,\tilde{\nu}]^{-}-[z,\tilde{\nu}]^{+}
\le 1-e^{-\Gamma(u)^{+}},
\]
again contradicting $1-e^{-s}<s$ for $s=\Gamma(u)^{+}>0$.

In both subcases we reach a contradiction. Therefore $\Gamma(u)^{+}=\Gamma(u)^{-}$
$\mathcal{H}^{N-1}$-a.e. on $U\cap\xi_{k}$, which means that $|D^{j}\Gamma(u)|(U\cap\xi_{k})=0$,
and hence $D^{j}\Gamma(u)=0$.
\smallskip

Finally, let $x\in J_{u}$ with approximate limits $0\le u^{-}(x)<u^{+}(x)$ and jump normal
$\nu_{u}(x)$. On the half-balls $\{y\in B_{r}(x):\pm(y-x)\cdot\nu_{u}(x)>0\}$, the rescaled
functions $y\mapsto u(x+ry)$ converge in measure as $r\to 0^{+}$ to $u^{+}(x)$ and $u^{-}(x)$,
respectively. Since $\Gamma:[0,\infty)\to[-\infty,+\infty)$ is continuous and strictly increasing,
we have $\Gamma(u^{-}(x))<\Gamma(u^{+}(x))$ (even when $\Gamma(u^{-}(x))=-\infty$), and
$y\mapsto\Gamma(u(x+ry))$ converges in measure on the two half-balls to the two distinct values
$\Gamma(u^{+}(x))$ and $\Gamma(u^{-}(x))$. Hence $x$ cannot be a Lebesgue point of $\Gamma(u)$,
that is, $J_{u}\subset S_{\Gamma(u)}$. Recalling that
$\mathcal{H}^{N-1}(S_{\Gamma(u)}\setminus J_{\Gamma(u)})=0$, we conclude that
\[
\mathcal{H}^{N-1}(J_{u})
\le\mathcal{H}^{N-1}(S_{\Gamma(u)}\setminus J_{\Gamma(u)})+\mathcal{H}^{N-1}(J_{\Gamma(u)})=0,
\]
and therefore $D^{j}u=0$. This concludes the proof.
\end{proof}
\bk

\subsection{Proofs completed}\label{3.4}

In this Section we shall provide the proof of Theorem \ref{teomain} that, as a consequence, will give us  that also Theorem  \ref{teoreg} holds.

\begin{proof}[Proof of Theorem \ref{teomain}]

 Let $u_p$ be a solution to \eqref{pbp} then it follows from Lemma \ref{corou} that there exists a bounded function $u\in BV(\Omega)$ such that  $u_p$ converges to $u$ in $L^q(\Omega)$ for every $q<\infty$ and $\nabla u_p$ converges to $Du$ locally *-weakly as measures\bk. Furthermore from Lemma \ref{corou} and from a weak lower semicontinuity argument one deduces that $\Gamma(u)\in BV_{\rm loc}(\Omega)$. 
We will carry on the proof by claims.
\medskip

{\it The term $h(u)f
	\in L^1_{\rm loc}(\Omega)$.}\medskip

Let us take a nonnegative $\varphi\in C^1_c(\Omega)$ as a test function in the distributional formulation of \eqref{pbp}. Then, getting rid of the nonnegative gradient term, Lemma \ref{lemmastima} and the Young inequality yield to
\begin{equation}\label{stimaL1loc}
\int_\Omega h_p(u_p)f\varphi \le \int_\Omega |\nabla u_p|^{p-2} \nabla u_p\cdot \nabla \varphi \le  \int_\Omega |\nabla u_p|^{p} + \int_\Omega |\nabla \varphi|^p \le C,
\end{equation}  
where $C$ is a positive constant independent of $p\in (1,p_0)$.
\\ The Fatou Lemma as $p\to 1^+$ in \eqref{stimaL1loc} gives 
\begin{equation}\label{l1loc}
\int_\Omega h(u)f\varphi \le C,
\end{equation} 
namely $h(u)f\in L^1_{\rm loc}(\Omega)$. We also underline that, in case $h(0)=\infty$, \eqref{l1loc} gives 
	\begin{equation*}
		\{u = 0\} \subset \{f = 0\},
	\end{equation*}
	up to a set of zero Lebesgue measure which implies that $u>0$ almost everywhere in $\Omega$ since $f>0$ almost everywhere in $\Omega$.
	\medskip
	
	{\it Existence of $z$.}\medskip

The existence of $z$ is standard and we recall it for the sake of completeness.\\
Let $1\le q<\frac{p}{p-1}$ then from Lemma \ref{lemmastima} and from the H\"older inequality one has
\begin{align}\label{stimaz}
\left(\int_\Omega ||\nabla u_p|^{p-2}\nabla u_p|^q\right)^{\frac{1}{q}} \le \left(\int_\Omega |\nabla u_p|^p\right)^{\frac{p-1}{p}}|\Omega|^{\frac{1}{q}-\frac{p-1}{p}}\leq C^{\frac{p-1}{p}}|\Omega|^{\frac{1}{q}-\frac {p-1}{p}}.		
\end{align}
This implies the existence of a vector field $z_q\in L^q(\Omega)^N$ such that $|\nabla u_p|^{p-2}\nabla u_p$ converges weakly to $z_q$ in $L^q(\Omega)^N$ and, through a diagonal argument, one obtains the existence of a unique vector field $z$, independent of $q$, such that $|\nabla u_p|^{p-2}\nabla u_p$ converges weakly to $z$ in $L^q(\Omega)^N$ for any $q<\infty$. Finally, taking $p\to 1^+$, one gets by weak lower semicontinuity in \eqref{stimaz} that $||z||_{L^q(\Omega)^N}\le |\Omega|^{\frac1q}$ and letting $q\to \infty$ one deduces $||z||_{L^\infty(\Omega)^N}\le 1$.
	
		\medskip
	
\noindent\emph{Distributional formulation (3.5), proof that $D^{j}u=0$, and identification of the vector field by (3.6).}

We first take a nonnegative $\varphi\in C_{c}^{1}(\Omega)$ as a test function in (3.14); after an
application of Young's inequality we use lower semicontinuity and the Fatou Lemma as $p\to 1^{+}$
in order to obtain that $z\in\mathcal{DM}^{\infty}_{\rm loc}(\Omega)$ and that
\begin{equation}\label{soprasol}
	-\mathrm{div}z \ge |D\Gamma(u)| + h(u)f \qquad \text{as measures in }\Omega,
\end{equation}
where we exploited that $\Gamma_{p,p}(u_{p})$ is locally bounded in $BV(\Omega)$ and
$\Gamma_{p,p}(u_{p})$ converges almost everywhere to $\Gamma(u)$.
\bk
\smallskip
Now let $0\le \varphi \in C^1_c(\Omega)$ and take $e^{\Gamma_{p,1}(u_p)} \varphi$ as a test function in the weak formulation of \eqref{pbp} obtaining after cancellations
\begin{equation}\label{peresp}
\begin{aligned}
	\int_{\Omega} |\nabla u_p|^{p-2}\nabla u_p\cdot \nabla \varphi e^{\Gamma_{p,1}(u_p)}
	= \int_\Omega h_p(u_p)fe^{\Gamma_{p,1}(u_p)}\varphi.
\end{aligned}	
\end{equation}
Hence taking $p\to 1^+$ one reaches to (observe that $e^{\Gamma_{p,1}(u_p)}\le C$)
\begin{equation}\label{peresp3}
\begin{aligned}
\int_{\Omega} z\cdot \nabla \varphi e^{\Gamma(u)} 
=  \int_\Omega h(u)fe^{\Gamma(u)}\varphi.
\end{aligned}	
\end{equation}
Indeed if $h(0)<\infty$ then one can simply pass to the limit in \eqref{peresp}. Hence, from here and in order to prove that \eqref{peresp3} holds, we assume that $h(0)=\infty$.
For the right hand side of \eqref{peresp} we write
\begin{equation}\label{rhs}
\int_{\Omega}h_p(u_p)f e^{\Gamma_{p,1}(u_p)} \varphi = \int_{\{u_p\le \delta\}}h_p(u_p)f e^{\Gamma_{p,1}(u_p)}\varphi + \int_{\{u_p> \delta\}}h_p(u_p)f e^{\Gamma_{p,1}(u_p)}\varphi,
\end{equation}
where $\delta>0$ is such that $\delta\not\in  \{k: |\{u=k \}|>0\}$ which is at most a countable set.
\\ We want to pass to the limit in \eqref{rhs} first as $p\to 1^+$ and then as $\delta \to 0$. One has that  
$$h_p(u_p)f e^{\Gamma_{p,1}(u_p)} \varphi\chi_{\{u_p> \delta\}} \le C\sup_{s\in [\delta,\infty)}[h(s)]\ f \varphi \in L^N(\Omega).$$
 Then one can apply the Lebesgue Theorem with respect to $p$ giving that
\begin{equation*}\label{rhs2}
\lim_{p\to 1^+}\int_{\{u_p> \delta\}}h_p(u_p)f e^{\Gamma_{p,1}(u_p)} \varphi= \int_{\{u> \delta\}}h(u)f e^{\Gamma(u)}\varphi.
\end{equation*}
Moreover the Young inequality and Lemma \ref{lemmastima} give that the left hand side of \eqref{peresp} is bounded with respect to $p$, then an application of the Fatou Lemma implies that $h(u)f e^{\Gamma(u)}\in L^1_{\rm loc}(\Omega)$. Hence the Lebesgue Theorem can be applied once more obtaining 
\begin{equation*}\label{rhs22}
\lim_{\delta\to 0}\lim_{p\to 1^+}\int_{\{u_p> \delta\}}h_p(u_p)fe^{\Gamma_{p,1}(u_p)}\varphi= \int_{\{u> 0\}}h(u)fe^{\Gamma(u)}\varphi.
\end{equation*}	
We are left to prove that the first term in the right hand side of \eqref{rhs} vanishes as $p\to 1^+$ and $\delta \to 0$. We take $V_{\delta}(u_p)\varphi$ ($V_{\delta}(s)$ is defined in \eqref{Vdelta}) as test function in the weak formulation of \eqref{pbp}, obtaining 
\begin{equation*}\label{lim1}
\begin{aligned}
\int_{\{u_p\le \delta\}}h_p(u_p)f \varphi\le \int_{\Omega}|\nabla u_p|^{p-2}\nabla u_p\cdot \nabla \varphi V_{\delta}(u_p),
\end{aligned}
\end{equation*}
and one gets
\begin{equation*}\label{lim2}
\begin{aligned}
\limsup_{p\to 1^+}\int_{\{u_p\le \delta\}}h_p(u_p)f \varphi \le \int_{\Omega}z\cdot \nabla \varphi V_{\delta}(u).
\end{aligned}
\end{equation*}
Now we take $\delta\to 0$
\begin{equation}\label{lim3}
\begin{aligned}
\lim_{\delta\to 0}\limsup_{p\to 1^+}\int_{\{u_p\le \delta\}}h_p(u_p)f \varphi \le \int_{\{u=0\}}z\cdot \nabla \varphi = 0,
\end{aligned}
\end{equation}
since $u>0$ almost everywhere in $\Omega$.
Let us also observe that \eqref{lim3} also gives that the first term in \eqref{rhs} vanishes in $\delta, p$ since
\begin{equation*}
\lim_{\delta\to 0}\limsup_{p\to 1^+}\int_{\{u_p\le \delta\}}h_p(u_p)fe^{\Gamma_{p,1}(u_p)}\varphi \le \lim_{\delta\to 0}\limsup_{p\to 1^+} C\int_{\{u_p\le \delta\}}h_p(u_p)f\varphi = 0.
\end{equation*}
Hence this implies 
\begin{equation*}
\lim_{p\to 1^+}\int_{\Omega}h_p(u_p)fe^{\Gamma_{p,1}(u_p)}\varphi = \int_{\{u>0\}}h(u)fe^{\Gamma(u)}\varphi = \int_{\Omega}h(u)fe^{\Gamma(u)}\varphi,
\end{equation*}
and that \eqref{peresp3} holds.\\
Observe that   $ze^{\Gamma(u)} \in \DM_{\rm loc}(\Omega)$ and \eqref{peresp3} implies 
\begin{equation}\label{zpere}
	-\operatorname{div}\left(ze^{\Gamma(u)} \right) = h(u)fe^{\Gamma(u)}  \ \ \text{as measures in }\Omega.
\end{equation}
Notice that $h(u)fe^{\Gamma(u)}\in L^{1}_{\rm loc}(\Omega)$ (since $h(u)f\in L^{1}_{\rm loc}(\Omega)$
and $e^{\Gamma(u)}\in L^{\infty}(\Omega)$). Combining \eqref{soprasol} and \eqref{zpere} with
$\|z\|_{L^{\infty}(\Omega)^{N}}\le 1$, we are now in position to apply Lemma \ref{lemmasalto} (case (i))
deducing that $D^{j}\Gamma(u)=0$ and $D^{j}u=0$. In particular, $e^{\Gamma(u)}$ has no jump part
and $|De^{\Gamma(u)}|=e^{\Gamma(u)}|D\Gamma(u)|$ as measures in $\Omega$.

Moreover one has that as measures
\begin{equation}\label{proofdistr}
\begin{aligned}
{e^{\Gamma(u)}} |D \Gamma(u)| &\stackrel{\eqref{soprasol}}{\le} {e^{\Gamma(u)}}\left(-\operatorname{div}z -h(u)f\right) \stackrel{\eqref{zpere}}{=} -{e^{\Gamma(u)}}\operatorname{div}z + \operatorname{div}\left(ze^{\Gamma(u)}\right) 
\\
&\stackrel{\eqref{dist1}}{=} \left(z, De^{\Gamma(u)}\right) \le | De^{\Gamma(u)}| = {e^{\Gamma(u)}} |D \Gamma(u)|,
\end{aligned}
\end{equation}
which gives that the first inequality in \eqref{proofdistr} is actually an equality, i.e.
$${e^{\Gamma(u)}}\left(-\operatorname{div}z -h(u)f\right) = {e^{\Gamma(u)}} |D \Gamma(u)| \ \text{as measures in }\Omega,$$
and since ${e^{\Gamma(u)}}\ge 1$ one has that
$$-\operatorname{div}z = |D \Gamma(u)| + h(u)f \ \text{as measures in }\Omega.$$
Finally applying Lemma \ref{chainrule}, recalling that $u>0$,  one has that $|D \Gamma(u)|=g(u)|Du|$ which gives that \eqref{def_distrp=1} holds. Furthermore it follows from Lemma \ref{lemmal1} that $z\in \mathcal{D}\mathcal{M}^\infty(\Omega)$. Moreover \eqref{proofdistr} also gives that
 $$\left(z, De^{\Gamma(u)}\right) = | De^{\Gamma(u)}| \ \text{as measures in }\Omega,$$
namely $\theta(z, De^{\Gamma(u)},x)=1$ \ $|D e^{\Gamma(u)}|$-\text{a.e. in } $\Omega$.
 Applying Proposition $2.2$ of \cite{mazonsegura} one gets that
 $$\theta(z, De^{\Gamma(u)},x) = \theta(z, D\Gamma(u),x) \ \ |D\Gamma(u)|-\text{a.e. in } \Omega,$$
 which implies that 
  $$\left(z, D{\Gamma(u)}\right) = | D{\Gamma(u)}| \ \text{as measures in }\Omega,$$
  since $|D{\Gamma(u)}|$ is absolutely continuous with respect to $|D e^{\Gamma(u)}|$.
With the same reasoning one has that
 $$\theta(z, Du,x) = \theta(z, D\Gamma(u),x) \ \ |D\Gamma(u)|-\text{a.e. in } \Omega,$$ 
 which also holds $|Du|-\text{a.e. in } \Omega$ since $g(s)>0$. This proves \eqref{def_zp=1}.
		
			\medskip
		{\it Boundary datum \eqref{def_bordop=1}.}
		\medskip
		
	Let us assume $m\geq 1$ and let us take $(\rho_\epsilon*u^m)\varphi$ as a test function in \eqref{def_distrp=1} where $0\le \varphi \in C^1_c(\Omega)$ and $\rho_\epsilon$ is a sequence of smooth mollifiers.
	Hence one gets 
	$$-\int_\Omega (\rho_\epsilon*u^m)\varphi \operatorname{div}z = \int_{\Omega} (\rho_\epsilon*u^m)g(u)\varphi |Du| + \int_{\Omega}h(u)f(\rho_\epsilon*u^m)\varphi,$$
	and recalling that $u\in BV(\Omega)\cap L^\infty(\Omega)$ one has that $(\rho_\epsilon*u^m)$ converges $\mathcal{H}^{N-1}$ a.e. to $(u^m)^*$ as $\epsilon\to 0$. Moreover $(u^m)^*\le ||u^m||_{L^\infty(\Omega)}$ and this allows to apply the Lebesgue Theorem passing to the limit with respect to $\epsilon$ each term of the previous.
	Hence, recalling that $|D^j u|=0$, one gets
	\begin{equation}\label{peru}
		- u^m \operatorname{div}z = u^m g(u)|Du| + h(u)fu^m \ \ \text{as measures in }\Omega.
	\end{equation} 
		Now we denote by
	$$\tilde{\Gamma}(s)=\int_0^s (m-g(t)t)t^{m-1} \ dt,  \qquad \tilde{\Gamma}_p(s)=\int_0^s (m-g_p(t)t)t^{m-1} \ dt, $$
	 and from now on  we take $m\ge m_0:=\sup_{s\in(0,\tilde{c}]}(g(s)s)$, which is finite thanks to \eqref{g1}. Since $g_p(t)t\le m_0$ for any $t\ge 0$ and any $p$ sufficiently close to $1$,  both integrands are nonnegative, so that $\tilde{\Gamma}$ and $\tilde{\Gamma}_p$ are nonnegative and nondecreasing. Let $0\le\phi\in C^1(\overline{\Omega})$;  we take $u_p^m\phi$ as a test function in \eqref{pbp} (recall that $u_p$  has  zero Sobolev trace) yielding to
	\begin{equation*}
		\begin{aligned}
			m\int_{\Omega}u_p^{m-1}|\nabla u_p|^p\phi +  \int_{\Omega}u_p^m|\nabla u_p|^{p-2}\nabla u_p\cdot\nabla\phi \bk = \int_{\Omega}g_p(u_p)|\nabla u_p|^p u_p^m\phi  + \int_\Omega h_p(u_p)fu_p^m\phi,
		\end{aligned}
	\end{equation*}
	which, after an application of the Young inequality  and recalling that $\tilde{\Gamma}_p(u_p)$ has zero trace on $\partial\Omega$,  gives
	\begin{equation}\label{bordo1}
		\begin{aligned}
			  \int_{\Omega}\phi|\nabla\tilde{\Gamma}_p(u_p)| +  \int_{\partial\Omega}\phi\,\tilde{\Gamma}_p(u_p)\, d \mathcal{H}^{N-1} & \le\int_\Omega h_p(u_p)fu_p^m\phi - \int_{\Omega}u_p^m|\nabla u_p|^{p-2}\nabla u_p\cdot\nabla\phi \\ & + \frac{p-1}{p}\int_{\Omega}(m-g_p(u_p)u_p)u_p^{m-1}\phi. 
		\end{aligned}
	\end{equation}
	Let us highlight that  the integrand of the last term in \eqref{bordo1} is bounded in $L^1(\Omega)$   with respect to $p$ thanks to Lemma \ref{corou}, so that this term vanishes as $p\to 1^+$. Moreover, since $0\le\tilde{\Gamma}_p(s)\le s^m$ and $u_p$ converges to $u\le\tilde{c}$ in $L^q(\Omega)$ for every $q<\infty$, one has that $\tilde{\Gamma}_p(u_p)$ converges to $\tilde{\Gamma}(u)$ in $L^1(\Omega)$. Then  one can take $p \to 1^+$ in \eqref{bordo1} by using, on the left hand side, the lower semicontinuity with respect to the $L^1(\Omega)$ convergence of the functional
	$$v\mapsto \int_\Omega \phi|Dv| + \int_{\partial\Omega}\phi|v|\, d\mathcal H^{N-1}, \ \ \ v\in BV(\Omega),$$
	and, on the right hand side,  the strong convergence of $h_p(u_p)u_p^m$ in $L^{\frac{N}{N-1}}(\Omega)$  together with the strong convergence of $u_p^m$ in $L^q(\Omega)$ and the weak convergence of $|\nabla u_p|^{p-2}\nabla u_p$ to $z$ in $L^{q}(\Omega)^N$ for every $q<\infty$,  as $p$ goes to $1$ for any fixed $m \geq m_0$. This argument takes to
	\begin{equation}\label{bound}
		\begin{aligned}
			\int_{\Omega}\phi |D \tilde{\Gamma}(u)| +  \int_{\partial\Omega} \phi\,\tilde{\Gamma}(u) d \mathcal{H}^{N-1} &\le\int_\Omega h(u)fu^m\phi - \int_\Omega u^mz\cdot\nabla\phi  \\ &\stackrel{\eqref{peru}}{=} - \int_\Omega u^m \phi \operatorname{div}z - \int_\Omega u^m g(u) \phi |Du|  - \int_\Omega u^mz\cdot\nabla\phi.
		\end{aligned}
	\end{equation}	
	An application of the chain rule on the left hand side and an application of the Green formula \eqref{green}   to $u^m\phi$, together with \eqref{f1} {(which can be applied since $\phi$ is continuous, so that $J_\phi=\emptyset$)},   on the right hand side of the previous leads, after cancellations, to
	\begin{equation}\label{bordo3}
		\begin{aligned}
			\int_{\Omega} \phi\bk|D u^m| +  \int_{\partial\Omega}\phi\,\bk\tilde{\Gamma}(u) d \mathcal{H}^{N-1} &\le \int_\Omega \phi\,\bk(z,D u^m) - \int_{\partial\Omega}\phi\,\bk u^m[z,\nu] d \mathcal{H}^{N-1} \\ &\le \int_{\Omega}\phi\bk|D u^m| - \int_{\partial\Omega}\phi\,\bk u^m[z,\nu] d \mathcal{H}^{N-1},
		\end{aligned}
	\end{equation}	
	where we have used that $[u^m\phi\,\bk z,\nu] =u^m\phi\,\bk[z,\nu]$ since $u^m\phi\bk$ belongs to $BV(\Omega)\cap L^\infty(\Omega)\bk$.
	Hence from \eqref{bordo3} one has that  
	$$\int_{\partial\Omega}\phi\bk\left(\tilde{\Gamma}(u) + u^m[z,\nu]\right) d \mathcal{H}^{N-1}= \int_{\partial\Omega} \phi\left(u^m(1+[z,\nu]) - \int_0^u g(t)t^m \ dt\right) d \mathcal{H}^{N-1} \le 0, \bk$$ 
	 for every $0\le\phi\in C^1(\overline{\Omega})$. Since the traces of such functions are dense in the nonnegative functions of $C(\partial\Omega)$, this implies that, for every integer $m\ge m_0$,
	\begin{equation}\label{bordo4}
		u^m(1+[z,\nu])\le \int_0^u g(t)t^m \ dt \le \sup_{s\in(0,\tilde{c}]}[g(s)s]\,\frac{u^{m}}{m} \ \ \ \mathcal{H}^{N-1}\text{-a.e. on } \partial\Omega,
	\end{equation}
	where the exceptional set can be chosen independent of $m$. Hence, \bk for  $\mathcal{H}^{N-1}$-a.e. \bk $x_0\in \partial\Omega$, one has either $u(x_0)=0$ or 
	$$ 0\le 1+[z,\nu](x_0)\le \frac{1}{m}\sup_{s\in(0,\tilde{c}]}[g(s)s], \bk$$
	and, taking $m\to\infty$, one deduces that $u(1+[z,\nu])=0$ $\mathcal{H}^{N-1}$ almost everywhere on $\partial\Omega$. This concludes the proof.
	\end{proof}

\section{The critical case $\theta =1$}\label{strongly}

In this section we analyze the critical case in which \eqref{g1}  is satisfied with $\theta=1$; again here in \eqref{h1} we consider $\gamma\le 1$.   
This case is critical in the sense that, in general, we lose coercivity.  In order to recover a priori estimates on the approximating solutions here we will need to further assume some control on the function $g$ and a stronger positivity of the datum $f$; the  interplay between $g$, $h$ and $f$ we shall consider  seems to be  not only technical  as it will be discussed below.

We  need to modify the definition of $\Gamma$ as follows: 
\begin{equation}\label{gamma1}
\Gamma(s)= \int_{1}^{s} g(t) dt.
\end{equation}

Observe that $\Gamma(s)$ defined by \eqref{gamma1} may blow up as $s$ approaches zero; prototypical example in the model case being a logarithm type growth.

 Our first additional assumption is the following: 
\begin{equation}\label{f}
\forall\omega\subset\subset \Omega \ \ \ \exists c_\omega>0: f\ge c_\omega>0 \text{ a.e. in } \omega.
\end{equation}
Moreover we ask the function $g$ to be somehow controlled by the function $h$ near zero. More precisely we assume
\begin{equation}\label{gh}
\displaystyle \limsup_{s\to 0} \frac{|\Gamma (s)|}{h(s)}<\infty.
\end{equation}
Let us just remark that \eqref{gh} implies that if $\Gamma$  blows up  at the origin (e.g. in the model case  $g(s)=s^{-1}$) then also  $h$ needs to. 

\medskip 
Our main result of this section is the following:
\begin{theorem}\label{teomain2}
	Let $f\in L^N(\Omega)$ satisfy  \eqref{f}, and let \eqref{g1}, \eqref{h1} and \eqref{gh}, with $\theta=1$ and $\gamma\le 1$, be in force. If for some $\delta>0$ one has that \begin{equation}\label{condiz} \displaystyle \max\left(\max_{s\in [0,\delta)} [g(s)s], ||f||_{L^N(\Omega)}\mathcal{S}_1h(\infty)\right)<1\,,\end{equation} then there exists a solution $u\in L^\infty(\Omega)$ to problem \eqref{pb} in the sense of Definition \ref{weakdefpositive}.
\end{theorem}
\begin{remark}\label{coerenza}
	Besides the smallness assumption on  $||f||_{L^N(\Omega)}\mathcal{S}_1h(\infty)$ which has been already discussed, assumption \eqref{condiz} is natural due to the possible criticality of the nonlinearity $g$.  If one thinks at the model case $g(s)=\lambda s^{-1}$,  the request reduces to $\lambda <1$ that allows us to  retrieve a sort of coercivity in the estimate. This type of assumption also appears, and is shown to be optimal,  in the case $p>1$ (see for instance \cite{ABLP, GPS2}). However, let us point out that, as $p\to 1^+$,  a curious  continuity break of this phenomenon comes out and, in some special cases,   solutions to problem \eqref{pb} can be constructed even beyond this threshold. This is also  related to the geometry of the set $\Omega$ and it will be discussed in Section \ref{break}. 
\end{remark}
We will work again  through the approximation process given by \eqref{pbp}. From here, in agreement with \eqref{gamma1}, the following notation is employed:
\begin{equation*}\label{gammap1}
\Gamma_{p,q}(s)= \displaystyle 
\int_{1}^{s} g_p^{\frac{1}{q}}(t) \ dt,
\end{equation*} 
which, if $q=1$ or $q =p$, converges to $\Gamma(s)$ for any $ s>0$ as $p\to 1^+$. 

\medskip

We have the following basic estimates in which, again, the strong positivity of $f$ is not needed:
\begin{lemma}\label{lemmastimastrong}
	Let $g$ and $h$ satisfy \eqref{g1} with $\theta=1$ and \eqref{h1} with $\gamma\le 1$, and let $0\le f\in L^N(\Omega)$ such that  $$\displaystyle \max\left(\max_{s\in [0,\delta)} [g(s)s], ||f||_{L^N(\Omega)}\mathcal{S}_1h(\infty)\right)<1\,,$$ for some $\delta>0$.  Let $u_p$ be a solution of \eqref{pbp} then there exists $p_0>1$ such that for any $\eta>0$
	\begin{equation}\label{stimanon}
	\int_{\Omega} |\nabla (e^{\eta u_p}-1)|^p \le C,
	\end{equation}
	for some constant $C$ which does not depend on $p\in(1,p_0)$; moreover, 
\begin{equation*}\label{stimabvcr}
|| e^{\eta u_p}-1||_{W^{1,1}_{0}(\Omega)}\leq C, 
\end{equation*}
again for some $C$ not depending on  $p\in (1,p_0)$. 
\end{lemma}
\begin{proof}
	Let us observe that only the behaviour in zero of $g$ is different with respect to the case of Lemma \ref{lemmastima}. Hence the boundedness of $e^{\eta G_k(u_p)}-1$ in $W^{1,p}_0(\Omega)$ for $k\ge \tilde{k}\ge \tilde{c}+1$ follows as before. Let us estimate the truncated functions;  take $T_\delta(u_p)$ as a test function in \eqref{pbp} obtaining (recall that after $\tilde{k}$ we have $g_p\equiv 0$ since we suppose $1<p<2$)
	\begin{equation*}\label{stimatkminore1}
	\begin{aligned}
	\int_{\Omega} |\nabla T_\delta(u_p)|^p &\le \int_\Omega g(u_p) |\nabla T_\delta(u_p)|^p u_p + \delta\int_\Omega g(u_p)|\nabla T_{\tilde{k}}(G_\delta(u_p))|^p  \\ &+ \left(c_2\delta^{1-\gamma} + \sup_{s\in[s_2,\infty)}h(s)\delta\right) \int_\Omega f
	\le \sup_{s\in [0,\delta)}[g(s)s]\int_\Omega |\nabla T_\delta(u_p)|^p \\ & + \delta \sup_{s\in [\delta,\tilde{c}+1)} g(s) \int_\Omega |\nabla T_{\tilde{k}}(G_\delta(u_p))|^p 
	+ \left(c_2\delta^{1-\gamma} + \sup_{s\in[s_2,\infty)}h(s)\delta\right) \int_\Omega f,
	\end{aligned}	
	\end{equation*}
	which gives 
	\begin{equation}\label{stimatkminore2}
	\begin{aligned}
	&\left(1-\sup_{s\in [0,\delta)}[g(s)s]\right)\int_{\Omega} |\nabla T_\delta(u_p)|^p \\ & \le \delta\sup_{s\in [\delta,\tilde{c}+1)}g(s)\int_\Omega |\nabla T_{\tilde{k}}(G_\delta(u_p))|^p  + \left(c_2\delta^{1-\gamma} + \sup_{s\in[s_2,\infty)}h(s)\delta\right) \int_\Omega f.
	\end{aligned}	
	\end{equation}	
	In order to estimate the first term on the right hand side of \eqref{stimatkminore2} we take $e^{\eta p T_{\tilde{k}}(G_\delta(u_p))}-1$ as a test function in \eqref{pbp} yielding to
	\begin{equation*}\label{stimatkminore3}
	\begin{aligned}
	\eta p\int_{\Omega} |\nabla T_{\tilde{k}}(G_\delta(u_p))|^pe^{\eta p T_{\tilde{k}}(G_\delta(u_p))} &\le \int_\Omega g_p(u_p) |\nabla T_{\tilde{k}}(G_\delta(u_p))|^p e^{\eta pT_{\tilde{k}}(G_\delta(u_p))} +  \sup_{s\in [\delta,\infty)}h(s) e^{\eta p \tilde{k}} \int_\Omega f,
	\\
	&\le  \sup_{s\in [\delta,\tilde{c}+1)} g(s) \int_\Omega |\nabla T_{\tilde{k}}(G_\delta(u_p))|^p e^{\eta p T_{\tilde{k}}(G_\delta(u_p))} +  \sup_{s\in [\delta,\infty)}h(s) e^{\eta p \tilde{k}} \int_\Omega f
	\end{aligned}	
	\end{equation*}
	and, requiring $\eta$ big enough, one has
	\begin{equation*}\label{stimatkminore4}
	\begin{aligned}
	\int_{\Omega} |\nabla T_{\tilde{k}}(G_\delta(u_p))|^pe^{\eta p T_{\tilde{k}}(G_\delta(u_p))} \le C,
	\end{aligned}	
	\end{equation*}
	which gathered in \eqref{stimatkminore2} and by monotonicity gives that \eqref{stimanon} holds for any $\eta>0$. Then,   reasoning as in the proof of Lemma \ref{lemmastima},  one concludes the proof.
\end{proof}  

Now we state  the existence of a bounded limit function $u$; its  proof closely follows the  one of Lemma \ref{corou} and we omit it.

\begin{lemma}\label{corou2}
	Under the assumptions of Lemma \ref{lemmastimastrong} there exists $p_0>1$ such that $ u_p$ is bounded in $BV(\Omega)$  uniformly with respect to $p\in (1,p_0)$. Moreover there exists $u\in BV(\Omega)\cap L^\infty(\Omega)$ such that $ u_p$ converges to $ u$ (up to a subsequence) in $L^q(\Omega)$ for every $q<\infty$ and $\nabla u_p$ converges to $D u$ locally *-weakly as measures. It holds
	\begin{equation*}\label{boundedstrong}
	||u||_{L^\infty(\Omega)}\le \tilde{c}.
	\end{equation*}
\end{lemma}	

We are ready to prove our main theorem.

\begin{proof}[Proof of Theorem \ref{teomain2}]
	
Let $u_p$ be a solution to \eqref{pbp} then from  Lemma \ref{corou2} one has that there exists $u\in BV(\Omega)\cap L^\infty(\Omega)$ such that $u_p$ converges to $u$ in $L^q(\Omega)$ for every $q<\infty$ and $\nabla u_p$ converges to $Du$ locally *-weakly as measures as $p$ tends to $1$. The construction of the bounded vector field $z$ then follows as in the proof of Theorem \ref{teomain}.

Observe that $h_p(u_p)f$ is locally bounded in $L^1(\Omega)$. Indeed, by simply taking a nonnegative $\varphi \in C^1_c(\Omega)$ and getting rid of the gradient term, one has through the Young inequality that
\begin{equation}\label{l1loc_coerstrong}
	\int_\Omega h_p(u_p)f\varphi \le \int_\Omega |\nabla u_p|^p + \int_\Omega |\nabla \varphi|^p\le C,
\end{equation} 
where the last inequality is a consequence of  \eqref{stimanon}. Moreover an application of the Fatou Lemma in \eqref{l1loc_coerstrong} with respect to $p$ gives $h(u)f\in L^1_{\rm loc}(\Omega)$ and, thanks to \eqref{f}, one also has that $h(u)\in L^1_{\rm loc}(\Omega)$. Moreover, even in this case, we underline that having $h(u)$ locally integrable implies that  $\{u=0\}$ is contained in the set $\{f=0\}$,  and so $u>0$. \bk
Now, for some $C>0$, one has  
$$|{\Gamma}_{p,p}(s)|\le |\Gamma(s)| + s + 1 \overset{\eqref{gh}}{\le } C h(s) +s +1$$ 
for any $0<s<\tilde{s}$ with $\tilde{s}$ sufficiently near to 0. Hence one has that ${\Gamma}_{p,p}(u_p)$ is bounded in $L^1_{\rm loc}(\Omega)$. Moreover taking a nonnegative $\varphi\in C^1_c(\Omega)$ as a test function in the weak formulation of \eqref{pbp}, getting rid of the nonnegative zero order term and applying the Young inequality, one yields to  
$$\int_{\Omega} |\nabla \Gamma_{p,p}(u_p)|\varphi \le \int_{\Omega} |\nabla u_p|^{p-2}\nabla u_p \cdot \nabla \varphi +\frac{p-1}{p} \int_\Omega \varphi \le \int_{\Omega} |\nabla u_p|^p + \int_{\Omega} |\nabla \varphi|^p + \frac{p-1}{p} \int_\Omega \varphi\le C,$$
which, once again thanks to Lemma \ref{lemmastimastrong}, gives that $\Gamma_{p,p}(u_p)$ is locally bounded in $BV(\Omega)$. 
\medskip

The proof that $h_p(u_p)f$ converges to $h(u)f$ locally in $L^1(\Omega)$ is identical to the one of the previous section and so we skip it.
We take $e^{ \Gamma_{p,1}(u_p)}\varphi$ ($\varphi\in C^1_c(\Omega)$) as a test function in \eqref{pbp} obtaining 
\begin{equation}\label{approxeqperesp}
\int_{\Omega} e^{ \Gamma_{p,1}(u_p)}|\nabla u_p|^{p-2}\nabla u_p\cdot \nabla \varphi = \int_{\Omega} h_p(u_p)fe^{\Gamma_{p,1}(u_p)}\varphi.
\end{equation}
As already done in the previous section one can prove that both terms in \eqref{approxeqperesp} converge,
obtaining that
$$
\int_{\Omega} e^{\Gamma(u)}z\cdot\nabla\varphi = \int_{\Omega} h(u)fe^{\Gamma(u)}\varphi.
$$
Now let us take a nonnegative $\varphi\in C^1_c(\Omega)$ as a test in \eqref{pbp} one yields to
$$\int_{\Omega} |\nabla u_p|^{p-2}\nabla u_p \cdot \nabla \varphi = \int_\Omega |\nabla {\Gamma}_{p,p}(u_p)|^p\varphi + \int_\Omega h_p(u_p)f\varphi.$$
Now we observe that ${\Gamma}_{p,p}(u_p)$ tends to $\Gamma(u)$ locally in $L^1(\Omega)$ as $p\to 1^+$ and applying the Young inequality, lower semicontinuity and the Fatou Lemma, one obtains 
\begin{equation*}
\int_{\Omega} z\cdot \nabla \varphi \ge \int_\Omega \varphi |D \Gamma(u)| + \int_\Omega h(u)f\varphi,
\end{equation*}
which is that $z\in \mathcal{D}\mathcal{M}^\infty_{\rm loc}(\Omega)$ and so
\begin{equation}\label{dis}
-\operatorname{div}z \ge |D \Gamma(u)| + h(u)f \text{  as measures in }\Omega.
\end{equation}

Let us note that, combining \eqref{dis} with $-\mathrm{div}(e^{\Gamma(u)}z)=h(u)fe^{\Gamma(u)}$ we proved above and the fact
$\|z\|_{L^{\infty}(\Omega)^{N}}\le 1$, we can apply Lemma \ref{lemmasalto} (case (i)) in order to deduce
that $D^{j}\Gamma(u)=0$ and $D^{j}u=0$. 
\bk
Furthermore one can reason as in \eqref{proofdistr} in order to deduce that inequality \eqref{dis} is actually an equality and that $(z,D\Gamma(u))=|D\Gamma(u)|$. Now one can apply Lemma \ref{chainrule} obtaining that \eqref{def_distrp=1} holds and then  \eqref{def_zp=1}. Moreover Lemma \ref{lemmal1} gives that $z\in\mathcal{D}\mathcal{M}^\infty(\Omega)$ and the proof of the fulfillment of the boundary datum realized   as in the proof of Theorem \ref{teomain}. 
\end{proof}

\section{Nonnegative data and strong singularities}
\label{nonnegative}

In this section we show how  the case of a purely nonnegative datum $f$ as well as the case of a possibly stronger zero order singularity, i.e. $\gamma>1$, can be treat. To simplify the exposition   the following useful notation is employed:
$$\sigma:= \max(1,\gamma)\,.$$

As already mentioned the case $h(0)<\infty$  is essentially the trivial one as, under suitable smallness assumptions on the data,  $u=0$ is a solution to problem \eqref{pbzero} and then \eqref{pb} (see Remark \ref{g=0}); nevertheless, even though this is not always the case,   we assume  $h(0)=\infty$ without loosing generality; the case $h(0)<\infty$ can be treat with straightforward modifications as in the previous sections.

Here is the suitable notion of solution in this general  case: 
\begin{defin}
	\label{weakdefstrong}
	A nonnegative $u\in BV_{\rm loc}(\Omega)$ such that  $\chi_{\{u>0\}}\in BV_{\rm loc}(\Omega)$ and $u^\sigma \in BV(\Omega)$ is a solution to problem \eqref{pb} if $g(u^*)\chi_{\{u>0\}}^\ast \in L^1_{\rm loc}(\Omega, |Du|),  h(u)f \in L^1_{\rm loc}(\Omega)$ and if there exists $ z\in \mathcal{D}\mathcal{M}^\infty_{}(\Omega)$ with $||z||_{L^\infty(\Omega)^N}\le 1$ such that 
	\begin{align}
	&(-\operatorname{div}z) \chi^{\ast}_{\{u>0\}}= g(u^*) \chi_{\{u>0\}}^\ast |Du| +  h(u)f \ \ \text{as measures in } \Omega, \label{def_distrp=1strongnon}
	\\
	&(\chi_{\{u>0\}}z,Du)=|Du| \label{def_zp=1strongnon} \ \ \ \ \text{as measures in } \ \Omega,
	\\
	& u^\sigma(x) (1+ [z,\nu] (x))=0  \label{def_bordop=1strongnon}\ \ \ \text{for  $\mathcal{H}^{N-1}$-a.e. } x \in \partial\Omega.
	\end{align}
\end{defin} 
\begin{remark}\label{rema} 
First observe that Definition \ref{weakdefstrong} is a general version of Definition \ref{weakdefpositive} allowing local regularity of the involved actors. The only global request, in the case $\gamma>1$,  is addressed to a suitable power of the solution $u$ yielding the well position  of  the boundary datum requirement \eqref{def_bordop=1strongnon}. 
	We underline that a solution as in Definition \ref{weakdefstrong} is also a solution in the sense of Definition \ref{weakdefpositive}  provided  $f>0$ a.e. in $\Omega$ and $\gamma\le 1$. 
	Indeed, since $h(0)=\infty$ having $h(u)f\in L^1_{\rm loc}(\Omega)$ implies that $u>0$  and then,  in equation \eqref{def_distrp=1strongnon} the characteristic functions disappear. Finally observe that this definition extends the one given in \cite{DGOP} in the case $g\equiv 0$.  
The presence of the function $g(u^*)\chi_{\{u>0\}}^\ast $	  in the previous definition is essentially technical since, again,  we do not request for the  solution $u$  to possess a purely diffuse derivative $D u$ (i.e. $D^j u =0$). A posteriori, since this is the case,  	the gradient term appearing in \eqref{def_distrp=1strongnon} can be intended as  	$g(u)\chi_{\{u>0\}} |Du|$. 		
\end{remark}
Here is our  existence theorem in the mild singular case (i.e. $\theta<1$); the proof will be sketched by highlighting the difference with the proof of Theorem \ref{teomain}. 
\begin{theorem}\label{teo_p>1strong}
	Let $0\le f\in L^N(\Omega)$ such that $||f||_{L^N(\Omega)}\mathcal{S}_1h(\infty)<1$ and let $g,h$ satisfy \eqref{g1} and \eqref{h1} with $\theta<1$. Then there exists a solution $u\in L^\infty(\Omega)$ to problem \eqref{pb} in the sense of Definition \ref{weakdefstrong}. {Moreover $D^j u =0$.}
\end{theorem}
\begin{proof}
	As already said, the proof of Theorem \ref{teo_p>1strong} adheres to  the one of Theorem \ref{teomain}; therefore,  we will only sketch the analogous  arguments while   major rigour will be provided  when the proofs are detaching each other. Clearly if $\gamma\le 1$  the estimates are the ones proved in Lemma \ref{lemmastima}. 
	Observe that the presence of a possibly strong singularity only affects the estimates when $u$ is small.  
	If $\gamma>1$,  Lemma \ref{lemmastima} can be reproduced by  treating the case $u\sim 0 $ as follows;  take $T_k^\gamma(u_p)$ ($k$ sufficiently small) as a test function in \eqref{pbp} obtaining 
	\begin{equation*}\label{stima1strong}
		\begin{aligned}
			&\gamma\int_{\Omega} |\nabla T_{k}(u_p)|^p T_k^{\gamma-1}(u_p) = \int_\Omega g_p(u_p) T_k^\gamma(u_p) |\nabla u_p|^p  \\ & + \int_\Omega h_p(u_p)fT_k^\gamma(u_p) 
			\le c_1k^{1-\theta}\int_\Omega |\nabla T_{k}(u_p)|^p T_k^{\gamma-1}(u_p)  \\ &+  \left(c_1k^{1-\theta}+\sup_{s\in[s_1,\infty)}g(s)k\right)\int_\Omega |\nabla T_{\tilde{k}-k}(G_{k}(u_p))|^p
			+\left(c_2 + \sup_{s\in[s_2,\infty)}h(s)k\right)\int_\Omega f,
		\end{aligned}	
	\end{equation*}
	which requiring $k\le \tilde{k}$ for some $\tilde{k}$ such that $\gamma - c_1\tilde{k}^{1-\theta}>C$ for some constant independent of $p$ and reasoning as in Lemma \ref{lemmastima} for the second term at the right hand side of the previous, one has 
	\begin{equation*}\label{stima2strong}
		\begin{aligned}
			\left(\frac{p}{\gamma-1+p}\right)^p\int_{\Omega} |\nabla T_{k}^{\frac{\gamma-1+p}{p}}(u_p)|^p  &\le C,
		\end{aligned}	
	\end{equation*}
	for any $k \le \tilde{k}$.
	Moreover, as in Lemma \ref{lemmastima}, $||e^{G_k(u_p)}-1||_{W^{1,p}_0(\Omega)} \le C$ for a constant independent of $p$ for any $k$ sufficiently large.

	\medskip 
	
	Local estimates are obtained by considering $0\le \varphi\in C^1_c(\Omega)$ and taking  $(T_k(u_p)-k)\varphi^p$ to test  \eqref{pbp},   deducing
	\begin{equation*}
		\label{stimaloc}
		\int_{\Omega}|\nabla T_k(u_p)|^p\varphi^p + p\int_\Omega |\nabla u_p|^{p-2}\nabla u_p\cdot \nabla \varphi \varphi^{p-1} (T_k(u_p)-k)\le 0,
	\end{equation*}   
	which gives, by Young's inequality 
	\begin{equation*}\label{stimalocgamma>1}
		\int_{\Omega} |\nabla T_k(u_p)|^{p}\varphi^p \le 	pk\int_{\Omega} |\nabla T_k(u_p)|^{p-1}|\nabla \varphi| \varphi^{p-1}\le  	pk\varepsilon \int_{\Omega} |\nabla T_k(u_p)|^{p}\varphi^{p} + pk C_\varepsilon\int_{\Omega} |\nabla \varphi|^p \,.
	\end{equation*}
	Then, requiring $\varepsilon$ small enough, it yields to
	\begin{equation*}
		||T_k(u_p)||_{W^{1,p}(\omega)}\le C, \ \ \ \forall \omega\subset\subset \Omega,
	\end{equation*}
	where $C$ is independent of $p$. 
	
	\medskip 
	Summarizing, one has that the following a priori estimates hold:
	\begin{equation}\label{prioristrong}
		\begin{cases}
			||e^{\eta u_p}-1||_{W^{1,p}_0(\Omega)}\le C,  \ \ &\text{if }\gamma\le 1,
			\\
			||e^{\eta u_p}-1||_{W^{1,p}(\omega)}\le C, \ \forall \omega\subset\subset \Omega, \ ||T_k^{\frac{\gamma-1+p}{p}}(u_p)||_{W^{1,p}_0(\Omega)}\le C, \ \text{for }   k\le \tilde{k} \bk  \ \ &\text{if }\gamma> 1,	
		\end{cases}
	\end{equation}
	where the constants $C$ do not depend on $p$ and for some $\tilde{k}>0$.
	Hence, estimates \eqref{prioristrong} imply that $e^{u_p}-1$ is locally bounded in $BV(\Omega)$. This allows to \textit{localize} Lemma \ref{corou} providing that
	there exists   $u\in BV_{\rm loc}(\Omega)\cap L^{\infty}(\Omega)$ such that, up to a subsequence, ${u_p}$ locally converges to ${u}$ in $L^q(\Omega)$ for $q<\infty$ and $\nabla u_p$ converges locally $\ast$-weakly as measures to $Du$. Moreover one has that $||u||_{L^\infty(\Omega)}\le \tilde{c}$ and  $u^\sigma\in BV(\Omega)$. Finally estimates \eqref{prioristrong} imply that, reasoning as in Lemma \ref{corou}, $\Gamma_{p,p}(u_p)$ is locally bounded in $BV(\Omega)$.
	
	\medskip
	The fact that $h(u)f$ belongs to $L^1_{\rm loc}(\Omega)$ follows exactly as in the proof of Theorem \ref{teomain}. Moreover the fact that the existence of the vector field $z$  (with $\|z\|_{L^\infty(\Omega)^N}\le 1$) \bk can be proved through the local estimate on $|\nabla u_p|^{p-2}\nabla u_p$; the definition of  the limit vector field $z$  can be extended  to the whole $\Omega$ by mean of  a standard diagonal argument; moreover,  $z\in \DM_{\rm{loc}}(\Omega)$.  Reasoning as in the proof of Theorem \ref{teomain}, \bk by lower semicontinuity and Fatou's Lemma one easily gets 
	\begin{equation}\label{soprasol-non}
		-\operatorname{div}z \ge  |D \Gamma(u)| +  h(u)f \ \ \text{as measures in }\Omega\,.
	\end{equation}
	Also observe that using Lemma \ref{lemmal1} one deduces that $z\in \DM_{\rm{}}(\Omega)$.
	\medskip 
	
	A relevant  main  difference with the case of a positive datum comes when one tries  to check  the weak formulation \eqref{def_distrp=1strongnon}  (as well as the absence of jumps $D^j u=0$) \bk due to the possible presence of $\chi_{\{u>0\}}^{\ast}$. 
	
	One has to show first  that $\chi_{\{u>0\}}\in BV_{\rm loc}(\Omega)$; to prove it,  set $S_\delta(s):=1-V_\delta(s)$, \bk let 
	$$
	K_{p,\delta}(s)=\int_0^s S'_\delta (t) e^{\Gamma_{p,1} (t)}\ d t  \quad\text{and}\quad K_\delta(s)=\int_0^s S'_\delta (t) e^{\Gamma (t)}\ d t \bk \,, 
	$$
	and consider  $ S_\delta (u_p) e^{\Gamma_{p,1} (u_p)}\varphi$ ($\varphi\in C^1_c(\Omega)$ and nonnegative) as a test function in the weak formulation of \eqref{pbp} deducing, after cancellations and an application of the Young inequality, that
	$$
	\begin{aligned}&\int_\Omega |\nabla K_{p,\delta} (u_p)|\varphi + \int_\Omega |\nabla u_p|^{p-2}\nabla u_p \cdot \nabla \varphi S_\delta (u_p) e^{\Gamma_{p,1} (u_p)} \\ &\le \int_\Omega h_p(u_p)fS_\delta (u_p) e^{\Gamma_{p,1} (u_p)}\varphi + \frac{p-1}{p} \int_\Omega S'_\delta(u_p)e^{\Gamma_{p,1} (u_p)}\varphi.\end{aligned}$$
	Hence  we  let  $p\to 1^+$ using  lower semicontinuity on the left hand side and  the Lebesgue Theorem on the other terms, obtaining, after rearranging   
	
	$$\int_\Omega  \varphi |D K_\delta (u)| \bk + \int_\Omega z \cdot \nabla \varphi S_\delta (u) e^{\Gamma (u)} \le \int_\Omega h(u)fS_\delta (u) e^{\Gamma (u)}\varphi\,.$$
	
	Now,  observe that $K_\delta(0)=0$ and, for every $s>0$ and $2\delta<s$, $K_\delta(s)=\frac{1}{\delta}\int_\delta^{2\delta} e^{\Gamma(t)}\,dt \to e^{\Gamma(0)}=1$ as $\delta\to 0$, so that $K_\delta(u)\to \chi_{\{u>0\}}$ in $L^1_{\rm loc}(\Omega)$. Moreover, \bk recall that $h(u)f \in L^1_{\rm loc}(\Omega)$ and so $\{u=0\}\subset\{f=0\}$;  thus   taking  $\delta \to 0$,  by lower semicontinuity on the first term and the Lebesgue Theorem on the others, \bk  one has
	$$\int_\Omega  \varphi |D \chi_{\{u>0\}}| \bk + \int_{\{u>0\}} z \cdot \nabla \varphi e^{\Gamma (u)} \le \int_\Omega h(u)fe^{\Gamma (u)}\varphi,$$
	that implies 
	$\chi_{\{u>0\}}\in BV_{\rm loc}(\Omega)$ since the second and the third term in the previous are finite. Let us also underline that, by \eqref{dist1},  this also implies that $\chi_{\{u>0\}}z \in \mathcal{D}\mathcal{M}^\infty_{\rm loc}(\Omega)$.
	\medskip

	We have then  proven 
	\begin{equation}\label{zperestrong}
		-\operatorname{div}\left(ze^{\Gamma(u)}\chi_{\{u>0\}} \right) +  |D \chi_{\{u>0\}}| \bk \le h(u)fe^{\Gamma(u)}  \ \ \text{as measures in } \Omega.
	\end{equation}
	 Combining \eqref{soprasol-non} and \eqref{zperestrong} (recalling that $\Gamma(0)=0$, $\|z\|_{L^\infty(\Omega)^N}\le 1$, and $h(u)fe^{\Gamma(u)}\in L^1_{\rm loc}(\Omega)$), we can now apply Lemma \ref{lemmasalto} (case (ii)) to deduce that $D^j\Gamma(u)=0$ and $D^j u=0$. In particular, since $\mathcal{H}^{N-1}(S_{\Gamma(u)})=0$, $\mathcal{H}^{N-1}$-a.e. point of the reduced boundary $\partial^{\ast}\{u>0\}$ is a Lebesgue point of $\Gamma(u)$, whence $\Gamma(u)^{\ast}=\Gamma(0)=0$ and $e^{\Gamma(u)^{\ast}}=1$ $\mathcal{H}^{N-1}$-a.e. on $\partial^{\ast}\{u>0\}$. Therefore $e^{\Gamma(u)}|D\chi_{\{u>0\}}|=|D\chi_{\{u>0\}}|$ as measures in $\Omega$, and moreover $\mathcal{H}^{N-1}(J_{e^{\Gamma(u)}}\cap J_{\chi_{\{u>0\}}})=0$, so that $(e^{\Gamma(u)}\chi_{\{u>0\}})^{\ast}=e^{\Gamma(u)}\chi^{\ast}_{\{u>0\}}$ $\mathcal{H}^{N-1}$-a.e. in $\Omega$ and \eqref{f1} can be applied to $e^{\Gamma(u)}$ and $\chi_{\{u>0\}}$. \bk

	Now we integrate  \eqref{soprasol-non} against $\chi^{\ast}_{\{u>0\}}$ in order to deduce that 
	\begin{equation}\label{soprasol2}
		(-\operatorname{div}z) \chi^{\ast}_{\{u>0\}} \ge  |D \Gamma(u)| + h(u)f \ \text{as measures in }\Omega\,.
	\end{equation}
	
	One has 
	
	\begin{equation}\label{proofdistrstrong}
		\begin{aligned}
			& {e^{\Gamma(u)}} |D \Gamma(u)| \stackrel{\eqref{soprasol2}}{\le}  { e^{\Gamma(u)}} \left((-\operatorname{div}z)\chi^{\ast}_{\{u>0\}} -h(u)f\right) \stackrel{\eqref{zperestrong}}{\leq} -{e^{\Gamma(u)}}\chi^{\ast}_{\{u>0\}}\operatorname{div}z + \diver (z e^{\Gamma(u)}\chi_{\{u>0\}}) 
			\\
			&-  e^{\Gamma (u)}  |D \chi_{\{u>0\}}|  = \left(z, D(e^{\Gamma(u)}\chi_{\{u>0\}})\right) -   e^{\Gamma (u)}  |D \chi_{\{u>0\}}| \stackrel{\eqref{f1}}{=} \left(\chi_{\{u>0\}}z, De^{\Gamma(u)}\right) 
			\\
			& +  e^{\Gamma(u)}\left(z, D\chi_{\{u>0\}}\right) -   e^{\Gamma (u)}  |D \chi_{\{u>0\}}| \le \left(\chi_{\{u>0\}}z, De^{\Gamma(u)}\right) \le  | D e^{\Gamma(u)}| =    {e^{\Gamma(u)}} |D \Gamma(u)|,
		\end{aligned}
	\end{equation}
	from which one deduces that \eqref{soprasol2} is  indeed an equality. By  Lemma \ref{chainrule} one has that  $|D \Gamma(u)|=g(u)\chi_{\{u>0\}}|Du|$  and then \eqref{def_distrp=1strongnon} holds. Moreover from \eqref{proofdistrstrong} one has
	$$\left(\chi_{\{u>0\}}z, De^{\Gamma(u)}\right) = |De^{\Gamma(u)}| \ \text{as measures in }\Omega,$$
	and reasoning as in the proof of Theorem \ref{teomain}, one also gets that \eqref{def_zp=1strongnon} holds.

	\medskip 
	
	Let us focus on the boundary datum; we can reason as in the proof of Theorem \ref{teomain} up to \eqref{bound}, with, by lower semicontinuity 
	$$
	- u^m \operatorname{div}z \geq  u^m g(u)|Du| + h(u)fu^m \ \ \text{as measures in }\Omega\,,
	$$
	in place of \eqref{peru}, as one starts from \eqref{soprasol-non}. 
	Thus one can get, for every $0\le\phi\in C^1(\overline{\Omega})$, \bk
	\begin{equation*}
		\begin{aligned}
			\int_{\Omega}\phi\bk|D \tilde{\Gamma}(u)| +  \int_{\partial\Omega} \phi\,\bk\tilde{\Gamma}(u) d \mathcal{H}^{N-1} \le - \int_\Omega u^m\phi\bk \operatorname{div}z - \int_\Omega u^m g(u)\phi\bk|Du|  - \int_\Omega u^mz\cdot\nabla\phi\bk,
		\end{aligned}
	\end{equation*}
	for  any $m\geq \max(2\sigma,m_0)$, where $m_0$ is the one defined in the proof of Theorem \ref{teomain} and \bk $\tilde{\Gamma}(s)=\int_0^s (m-g(t)t)t^{m-1} \ dt$  (recall that $u^m\in BV(\Omega)$ as $m\ge\sigma$ and $u^\sigma\in BV(\Omega)\cap L^\infty(\Omega)$)\bk. Now applying the chain rule formula at the left hand side and the Green formula at the right hand side one yields to
	\begin{equation*}\label{bordo3strong}
		\begin{aligned}
			\int_{\Omega}\phi\bk|D u^m| +  \int_{\partial\Omega}\phi\,\bk\tilde{\Gamma}(u) d \mathcal{H}^{N-1} \le \int_\Omega \phi\,\bk(z,D u^m) - \int_{\partial\Omega}\phi\,\bk u^m [ z,\nu] d \mathcal{H}^{N-1} \le\bk \int_{\Omega}\phi\bk|D u^m| - \int_{\partial\Omega}\phi\,\bk u^m [ z,\nu] d \mathcal{H}^{N-1}.
		\end{aligned}
	\end{equation*}	
	Hence
	$$ \int_{\partial\Omega}\phi\left(u^m(1+[ z,\nu]) - \int_0^u g(t)t^m \ dt\right) d \mathcal{H}^{N-1} \le 0, \bk$$ 
	 for every $0\le\phi\in C^1(\overline{\Omega})$. Reasoning as at the end of the proof of Theorem \ref{teomain}, the \bk previous implies that for  $\mathcal{H}^{N-1}$-a.e. \bk $x_0\in \partial\Omega$ either  $u(x_0)=0$ or 
	$$ 0\le 1+[ z,\nu](x_0) \le \frac{1}{m}\sup_{s\in(0,\tilde{c}]}[g(s)s], \bk$$
	which, after taking $m\to\infty$, implies that $u^\sigma(1+[ z,\nu]) =0$ $\mathcal{H}^{N-1}$-almost everywhere on $\partial\Omega$. This concludes the proof.	   
	\end{proof}
	
	Let us point out some facts on the critical case  $\theta=1$.  A straightforward adjustment   of the proof of  Theorem \ref{teomain2} can be done in order to include the   strongly singular case $\gamma>1$ the result being the following
	\begin{theorem}
	Let $f\in L^N(\Omega)$ satisfy \eqref{f} and let $g,h$ satisfy \eqref{g1}, \eqref{gh} and \eqref{h1} with $\theta=1$ and $\gamma> 1$. If for some $\delta>0$ one has that $$\displaystyle \max\left(\max_{s\in [0,\delta)} [g(s)s], ||f||_{L^N(\Omega)}\mathcal{S}_1h(\infty)\right)<1\,,$$ then there exists a solution $u\in L^\infty(\Omega)$ to problem \eqref{pb} in the sense of Definition \ref{weakdefstrong}.  Moreover $D^j u =0$. \bk
	\end{theorem}

\begin{remark}\label{5.5}
Let us  stress that the  critical case  $\theta=1$  is much more delicate in   presence of a merely nonnegative datum $f$ not satisfying \eqref{f}.  In fact the set $\{u=0\}$ is not precluded  to have   positive measure so that the function $\Gamma(u)$ introduced by \eqref{gamma1} is not even defined, nor, a fortiori, the term   $g(u)|D u|$. 

Another remark is in order  on assumption \eqref{gh} since, if it is not in force,  then an  highly degenerate behavior of the approximating solutions $u_p$  could appear.   Let us think for instance at the  the case $h(0)<\infty$ and $g(s)\sim s^{-1} $ (notice that in this case assumption \eqref{gh} is not satisfied);   as we already mentioned (see Remark \ref{czero}) it is possible to show that, for data $f$ small enough,  then the bound $\tilde c$ given in \eqref{lim} is zero.  This implies that the approximating solutions $u_p\to 0$ a.e. on $\Omega$.  \bk
\end{remark}

\section{Remarks and examples}\label{rae}

\subsection{Breaking of nonexistence phenomenon for $p=1$} \label{break} Let $p>1$, consider $g(s)= \lambda s^{-1}$,  and $h(s)=s^{-\gamma}$ ($\gamma\ge 0$). We show that, if $\displaystyle \lambda\ge 1$, it can not exist $u$ solution to
\begin{equation}\label{noex}
-\Delta_p u = \frac{\lambda|\nabla u|^p}{u} + \frac{1}{u^\gamma}.
\end{equation}
By a solution to \eqref{noex} we mean a function $u\in W^{1,p}_0(\Omega)\cap L^\infty(\Omega)$ satisfying
$$\int_{\Omega}|\nabla u|^{p-2}\nabla u\cdot \nabla \varphi = \lambda\int_\Omega\frac{|\nabla u|^p}{u}\varphi + \int_{\Omega} \frac{\varphi}{u^\gamma},$$ 
for every $\varphi \in W^{1,p}_0(\Omega)\cap L^\infty(\Omega)$.
First of all observe that having $u^{-\gamma}\in L^1_{\rm loc}(\Omega)$ implies that $u>0$ almost everywhere in $\Omega$.
Moreover one can take $\varphi=u$ deducing that 
$$\int_{\Omega}|\nabla u|^p = \lambda \int_{\Omega} |\nabla u|^p + \int_\Omega u^{1-\gamma},$$
i.e.
$$\int_\Omega u^{1-\gamma}=0,$$
which is clearly a contradiction.

If $p=1$, making analogous  calculation one realizes  a correspondent nonexistence  instance for solutions with zero trace at the boundary; in fact,  assume  that, for $\lambda \ge 1$, there exist $u\in BV(\Omega)\cap L^\infty(\Omega)$ and $z\in \DM(\Omega)$, such that  $(z,Du)=|Du|$ and  
\begin{equation*}
-\operatorname{div}z = \frac{\lambda|D u|}{u} + \frac{1}{u^\gamma}.
\end{equation*}
Reasoning as in the proof of \eqref{peru}  it is possible to test the previous by $u$ obtaining
$$-\int_{\Omega}u\operatorname{div}z = \lambda \int_{\Omega} |D u| + \int_\Omega u^{1-\gamma},$$
which after an application of the Green formula takes to
$$\int_{\Omega}(z,Du) + \int_{\partial \Omega} u \ d\mathcal{H}^{N-1} =\int_{\Omega}|Du| + \int_{\partial \Omega} u \ d\mathcal{H}^{N-1}= \lambda \int_{\Omega} |D u| + \int_\Omega u^{1-\gamma},$$
which gives
$$\int_{\partial \Omega} u \ d\mathcal{H}^{N-1}  \ge  \int_\Omega u^{1-\gamma},$$
which is a contradiction if $u(x)=0$ for $\mathcal{H}^{N-1}$ almost every $x\in \partial \Omega$.

\medskip
\begin{remark}
 Unluckily one has to observe that the occurrence of the case  $u(x)=0$ for $\mathcal{H}^{N-1}$-a.e. $x$ on $\partial\Omega$ represents a quite rare case. In Example \ref{example2} below  we shall see a generic enough class of solutions whose boundary datum is assumed pointwise only at non-regular points of $\partial\Omega$. Observe that if $u$ is a non-trivial constant solution of the homogeneous Dirichlet problem associated to 
$$
-\Delta_1 u  =  {u^{-\gamma}}\,,
$$
then one trivially gets a solution to the homogeneous Dirichlet problem associated to 
$$
-\Delta_1 u  = \frac{\lambda|D u|}{u^\theta} + {u^{-\gamma}}\,;
$$
that is, in contrast with the case $p>1$ (with $\theta=1$), one should have existence of solutions beyond the threshold $\lambda=1$ (and, by the way, for any positive $\theta$). In Example \ref{example1} below we construct such  constant solutions.  
\bk\end{remark}

\subsection{Constant vs nonconstant solutions}
In the following example we show that, in certain model cases explicit  constant (non-trivial) solutions of problem \eqref{pb} can be found; it consists in a suitable re-interpretation of an example given in \cite{DGOP}. We first need the following 
\begin{defin}
		A bounded convex set $E$ of  class $C^{1,1}$  is said to be  {calibrable} if there exists a vector field $\xi\in L^\infty(\R^N, \R^N)$ such that $\|\xi\|_\infty\leq 1$, $(\xi,D\chi_E)=|D\chi_E|$ as measures,  and $$-{\rm div}\xi=\lambda_E\chi_E \quad {\rm in \ } \mathcal D'(\R^N) $$ for some constant $\lambda_E$. In this case  $\lambda_E=\frac{Per(E)}{|E|}$ and  $[\xi,\nu^E]=-1$, $\mathcal H^{N-1}$-a.e in $\partial E$ (see \cite[Section 2.3]{acc} and \cite{MP}).
	\end{defin}
	
There is plenty of calibrable sets, for instance  if $E=B_R(0)$, for some $R>0$, then $E$ is calibrable. More in general  	    a bounded and convex set $E$ is calibrable if and only if the following condition holds: $$(N-1)\|{\mathcal H}_E\|_{L^\infty(\partial E)}\leq \lambda_E=\frac{ Per(E)}{|E|},$$ where ${\mathcal H}_E$ denotes the ($\mathcal H^{N-1}$-a.e. defined) mean curvature of $\partial E$ (\cite[Theorem 9]{acc}).

\begin{example}
\label{example1} 
Let  $\theta, \gamma$, and  $\lambda$ be  three positive parameters and consider the following 
 	\begin{equation}\label{lm}
	\begin{cases}
	\displaystyle - \Delta_1 u=  \frac{\lambda |D u|}{u^\theta} +{u^{-\gamma}} &  \text{in}\, \Omega, \\
	u=0 & \text{on}\ \partial \Omega
	\,.
	\end{cases}
	\end{equation}
If $\Omega$ is a calibrable set, let us prove that $u=\left(\frac{|\Omega|}{Per(\Omega)}\right)^{\frac{1}{\gamma}}$ is a  solution to \eqref{lm}. It suffices to take the restriction to $\Omega$ of the vector field in the definition of calibrability; i.e.: $z:=\xi_{\res_\Omega}$. In fact, as $D u =0$,using  the properties of  $\xi$ one has
		\begin{equation}\label{usi}
		-{\rm div}z=\frac{Per(\Omega)}{|\Omega|}=u^{-\gamma}\ \ \ \text{and}\ \ \  [\xi,\nu^{\Omega}]=-1\,.
		\end{equation}
		
		Moreover, using both \eqref{green} \bk and \eqref{usi}, one has 
		$$\begin{array}{l}\displaystyle  (z,Du)(\Omega)=\int_\Omega \left(\frac{|\Omega|}{Per(\Omega)}\right)^{\frac{1}{\gamma}}\frac{Per(\Omega)}{|\Omega|}\,dx\\\\\displaystyle +\int_{\partial\Omega}[\xi,\nu^\Omega]\left(\frac{|\Omega|}{Per(\Omega)}\right)^{\frac{1}{\gamma}}\,d\mathcal H^{N-1}=0= |Du|(\Omega)\,;\end{array}$$
\end{example}
which implies the desired result.  

One may think that the previous example is generic enough in order to trivialize \eqref{pb} (at list in model cases). The following example of non-constant solutions to \eqref{pb} shows that this is not the case.

\begin{example}
\label{example2} Let us show a situation in which the unique solution $u$ to the following problem 
\begin{equation}
\label{pbe}
\begin{cases}
\dis -\Delta_1 u = {u^{-\gamma}} & \text{in}\;\Omega,\\
u=0 & \text{on}\;\partial\Omega,
\end{cases}
\end{equation}  
is, for $\gamma>0$, not constant; as we will see,  this fact will lead to a non-constant solution  of problem involving the gradient term. 
Let $\Omega$ be a convex open set and $H_{\Omega} (x)$ be the {variational mean curvature} of $\Omega$ (see \cite{BGT} for details). In \cite{MP} it is shown that $-H_{\Omega} (x)$ is a (so called) large solution to $\Delta_1 v= v $, i. e. 
\begin{equation}\label{large}			\begin{cases}
\displaystyle \Delta_1 v= v &  \text{in}\, \Omega, \\
v=\infty & \text{on}\ \partial \Omega\,.			\end{cases}
\end{equation}
Without entering into technicalities, only recall that $|| H_\Omega ||_{L^{\infty}(\rn)}< \infty$ if and only if $\Omega$ is of class $C^{1,1}$; 
in fact, if $\Omega$ is of class $C^{1,1}$ then it satisfies the uniform interior ball condition and so, the (unique) large solution of \eqref{large} is bounded	(\cite[Theorem 4.2]{MP}).  Viceversa, if $|| H_\Omega ||_{L^{\infty}(\rn)}< \infty$ then $\Omega$ is of class $C^{1,1}$ (\cite[Theorem 4.4]{MP}). 
In particular,  these  solutions are locally bounded and they assume the (large) datum $\infty$  at non-regular points of $\Omega$ (e.g. at corners).  		
Through the change of variable $u=v^{-\frac{1}{\gamma}}$ problem \eqref{large} formally transforms into \eqref{pbe} then one retrieves  that solutions to  problem \eqref{pbe} may be non-constant. In fact, in general the $H_\Omega (x)$ is known to be non-constant  if the set is not calibrable (see \cite{BGT,acc}); for instance if $\Omega$ is not $C^{1,1} $ (say a square), then  $u=v^{-\frac{1}{\gamma}}$ is positive everywhere and it attains the value $0$ only at the corners of $\Omega$. 
 \begin{figure}[htbp]\centering
\includegraphics[width=2in]{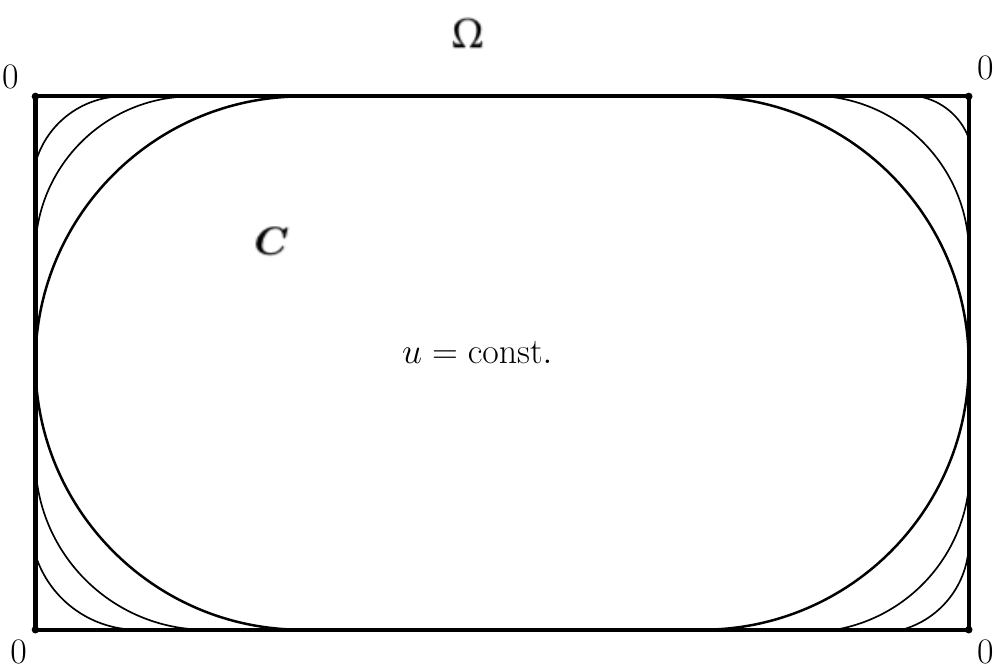}
\caption{A non-constant solution $(-H_\Omega(x))^{-\frac{1}{\gamma}}$ to \eqref{pbe}} \label{calo}
\end{figure}
Also observe that $u$ is constant inside the (unique) Cheeger $C$ set contained in $\Omega$, this constant being  nothing but a suitable power of the Cheeger constant of $C$ (see Figure \ref{calo}).

 \end{example}

The previous example of non-constant solution to \eqref{pbe} infers the nontriviality (even in the model case) to our problem \eqref{pb}: in fact,   assume   by contradiction  that $w$, solution to the Dirichlet problem associated to $$\dis -\Delta_1 w = g(w)|D w|+ w^{-\gamma} \ \text{in}\;\Omega,$$ is constant;  then $w$ is another solution to \eqref{pbe} which, by uniqueness (\cite[Theorem 3.5]{DGOP}), would give a contradiction  (the same argument applies for a smooth decreasing $h$).


\begin{thebibliography}{10}
	\bibitem{ds} B. Abdellaoui, A. Dall'Aglio and S. Segura de Le\'on, Multiplicity of solutions to elliptic problems involving the 1-Laplacian with a critical gradient term, Adv. Nonlinear Stud. { 17} (2), 333-353  (2017)
	
	\bibitem{acc} F. Alter,  V. Caselles and A.  Chambolle, \newblock {A characterization of convex calibrable sets in $\R^N$, }
\newblock  {Math. Ann.} { 332}  329-366 (2005) 
	\bibitem{acm} {L. Ambrosio, G. Crippa and S. Maniglia, Traces and fine properties of a $BD$ class of vector fields and applications, Ann. Fac. Sci. Toulouse Math. (6) 14 (4), 527-561 (2005)}
	\bibitem{afp} L. Ambrosio, N. Fusco and D. Pallara, Functions of Bounded Variation and Free Discontinuity Problems, Oxford Mathematical Monographs, 2000
		\bibitem{ABCM} F. Andreu, C. Ballester, V. Caselles and J. M. Maz\'on,
	{The Dirichlet problem for the total variation flow.}
	J. Funct. Anal. { 180} (2), 347-403 (2001) 
\bibitem{ACM}  F. Andreu, V. Caselles and J. M. Maz\'on,
	{Parabolic quasilinear equations minimizing linear growth functionals.}
	Progress in Mathematics, 223, Birkh\"auser Verlag, Basel, 2004	



\bibitem{ads} F. Andreu, A. Dall'Aglio and S. Segura de Le\'on, Bounded solutions to the 1-Laplacian equation with a critical gradient term, Asymptotic Analysis   80 (1-2), 21-43  (2012)
	\bibitem{An} G. Anzellotti, {Pairings between measures and bounded functions and compensated compactness,} Ann. Mat. Pura Appl. { 135} (4),  293-318 (1983)
	
	\bibitem{a6}
D. Arcoya, J. Carmona, T. Leonori, P. J. Mart\'nez-Aparicio, L. Orsina and F. Petitta,
Existence and nonexistence of solutions for singular quadratic quasilinear equations,
J. Differential Equations  246, 4006-4042  (2009)


\bibitem{ABLP} D. Arcoya, L. Boccardo, T. Leonori and A. Porretta: 
Some elliptic problems with singular natural growth lower order terms.
J. Differential Equations 249, (11), 2771-2795 (2010)

	\bibitem{BGT} E. Barozzi, E. Gonzalez, and I. Tamanini, {The mean curvature of a set of finite perimeter,} Proc. Amer. Math. Soc. 99,  313-316 (1987)
\bibitem{BCRS} M. Bertalmio, V. Caselles, B. Roug\'e and  A.  Sol\'e, 
TV based image restoration with local constraints, 
Special issue in honor of the sixtieth birthday of Stanley Osher, 
J. Sci. Comput. { 19} (1-3), 95-122 (2003)
\bibitem{b1}
L. Boccardo,
Dirichlet problems with singular and gradient quadratic lower order terms,
ESAIM Control Optim. Calc. Var. 14 , 411-426 (2008)
\bibitem{bmp} L. Boccardo, F. Murat and  J. P. Puel, Existence des solutions faibles des \'equations elliptiques quasi-lineaires \`a  croissance quadratique, in: H. Brezis, J.L. Lions (Eds.), Nonlinear P.D.E. and Their Applications, Coll\`ege de France Seminar, vol. IV, in: Research Notes in Mathematics, vol. 84, Pitman, London, 1983, 19-73

\bibitem{bmp2} L. Boccardo, F. Murat and  J. P. Puel, Existence of bounded solutions for nonlinear elliptic unilateral problems, Ann. Mat. Pura Appl. 152, 183-196 (1988) 



 	\bibitem{BO} L. Boccardo and L. Orsina, {Semilinear elliptic equations with singular nonlinearities}, {Calc. Var. Partial Differential Equations} { 37}, 363-380 (2010)
	\bibitem{C} V. Caselles, On the entropy conditions for some flux limited diffusion equations,  J. Differential Equations 250,  3311-3348 (2011)
\bibitem{CF} G. Q. Chen and H. Frid, Divergence-measure fields and hyperbolic conservation laws, Arch. Ration. Mech. Anal. 147 (2),  89-118 (1999)
     \bibitem{CT} M. Cicalese and C. Trombetti, {Asymptotic behaviour of solutions to $p$-Laplacian equation,} Asymptotic Analysis 
 { 35}, 27-40 (2003)
 \bibitem{crt} M. G. Crandall, P. H. Rabinowitz and L. Tartar, {On a dirichlet problem with a singular nonlinearity}, Comm. Part. Diff. Eq. 2, 193-222 (1977)
 \bibitem{cdc} {G. Crasta and V. De Cicco, Anzellotti's pairing theory and the Gauss--Green theorem, Adv. Math. 343, 935-970 (2019)}
\bibitem{dop} A. Dall'Aglio, L.  Orsina and F.  Petitta, 
Existence of solutions for degenerate parabolic equations with singular terms, 
Nonlinear Anal. 131, 273-288 (2016)

\bibitem{ds2} A. Dall'Aglio  and  S. Segura de Le\'on, Bounded solutions to the 1-Laplacian equation with a total variation term,  Ricerche mat (2018), in press, doi:10.1007/s11587-018-0425-5
 \bibitem{DCA} L. De Cave, {Nonlinear elliptic equations with singular nonlinearities,} Asymptotic Analysis { 84} (3-4),  181-195 (2013)	
 \bibitem{dgs}  V. De Cicco, D. Giachetti and S. Segura de Le\'on, Elliptic problems involving the 1-Laplacian and a singular lower order term,  J. Lond. Math. Soc. (2) 99 (2), 349-376 (2019)
\bibitem{DGOP} V. De Cicco, D. Giachetti, F. Oliva and  F. Petitta, {The Dirichlet problem for singular elliptic equations with general nonlinearities}, Calc. Var. Partial Differential Equations, {58}(4) (2019) 
\bibitem{D} F. Demengel, On some nonlinear partial differential equations involving the "1''-Laplacian and critical Sobolev exponent, ESAIM Control Optim. Calc. Var. { 4}, 667-686 (1999)
\bibitem{DG}	P.  Donato and  D,  Giachetti, Existence and homogenization for a singular problem through rough surfaces, SIAM J. Math. Anal. { 48} (6), 4047-4086 (2016)
\bibitem{d} R. Durastanti, Asymptotic behavior and existence of solutions for singular elliptic equations, Ann. Mat. Pura e Appl., in press, doi:10.1007/s10231-019-00906-0

\bibitem{fm} V. Ferone and  F. Murat, Nonlinear problems having natural growth in the gradient: an existence result when the source terms are small, Nonlinear Anal. 42 1309-1326 (2000) 

\bibitem{fp} G. M. Figueiredo and  M.T.O. Pimenta,
Sub-supersolution method for a quasilinear elliptic problem involving the 1-laplacian operator and a gradient term,  J. Funct. Anal., in press, doi: 10.1016/j.jfa.2019.108325.

\bibitem{GMM2} D. Giachetti, P. J. Mart\'inez-Aparicio and F. Murat, A semilinear elliptic equation with a mild singularity
at $u = 0$, Existence and homogenization,  J. Math. Pures Appl. { 107},  41-77 (2017)
\bibitem{GMM} D. Giachetti, P. J. Mart\'inez-Aparicio and F. Murat,
	{Definition, existence, stability and uniqueness of the solution to a semilinear elliptic problem with a strong singularity at $ u = 0 $,} Ann. Scuola Normale Pisa (5) { 18} (4), 1395-1442 (2018)
	\bibitem{gps}
D. Giachetti, F. Petitta and  S. Segura de Le\'on,
Elliptic equations having a singular quadratic gradient term and a changing sign datum,
Commun. Pure Appl. Anal. 11 , 1875-1895 (2012)
\bibitem{GPS2} D. Giachetti, F. Petitta and  S.  Segura de Le\'on, A priori estimates for elliptic problems with a strongly singular gradient term and a general datum,  Differential Integral Equations { 26}  (9-10), 913-948 (2013)
\bibitem{gmp} L. Giacomelli, S. Moll and F.  Petitta,  Nonlinear diffusion in transparent media: the resolvent equation, Adv. Calc. Var. { 11}  (4), 405-432 (2018)
\bibitem{gt} N. Grenon, C. Trombetti, Existence results for a class of nonlinear elliptic problems with $p$-growth in the gradient, Nonlinear Anal. 52 931-942  (2003)
\bibitem{HL} G. Huisken and T. Ilmanen, The Inverse Mean Curvature Flow and the Riemannian Penrose Inequality,  J. Differential Geom. 59  353-438 (2001)
    \bibitem{K}  B. Kawohl, {On a family of torsional creep problems,}
     J. Reine Angew. Math. { 410}, 1-22 (1990)
     \bibitem{ka} B. Kawohl,  {From $p$-Laplace to mean curvature operator and related questions,}  Progress in Partial Differential Equations: the Metz Surveys, Pitman Res. Notes Math. Ser., Vol. 249, Longman Sci. Tech., Harlow, pp. 40-56, 1991
  \bibitem{KS}  B.  Kawohl and F.  Schuricht, {Dirichlet problems for the $1$-Laplace operator, including the eigenvalue problem}, Commun. Contemp. Math. { 9} (4),  515-543 (2007)
\bibitem{ls} M. Latorre and S. Segura de Le\'on, Elliptic equations involving the $1$-Laplacian and a total variation term with $L^{N,\infty}$-data, Atti Accad. Naz. Lincei Rend. Lincei Mat. Appl. 28  (4), 817-859 (2017)
\bibitem{LM} A. C. Lazer and P. J. McKenna, {On a singular nonlinear elliptic boundary-value problem}, {Proc. Amer. Math. Soc.} { 111},  721-730 (1991)
\bibitem{mazonsegura} J. M. Maz\'on and S. Segura de Le\'on, The Dirichlet problem for a singular elliptic equation arising in the level set formulation of the inverse mean curvature flow, Adv. Calc. Var. 6  (2), 123-164 (2013)
\bibitem{MST1}  A. Mercaldo, S. Segura de Le\'on and C. Trombetti, 
	{On the behaviour of the solutions to $p$-Laplacian equations as $p$ goes to $1$,}
	 Publ. Mat. { 52},  377-411 (2008)
\bibitem{MST2} A. Mercaldo, S. Segura de Le\'on and C. Trombetti, On the solutions to $1$-Laplacian equation with $L^1$ data,   J. Funct. Anal. 256 (8), 2387-2416 (2009)
\bibitem{M} Y. Meyer,  {Oscillating Patterns in Image Processing and Nonlinear Evolution Equations: The Fifteenth Dean Jacqueline B. Lewis Memorial Lectures},   Providence, RI: American Mathematical Society, 2001
\bibitem{MP} S. Moll and F. Petitta, {Large solutions for the elliptic 1-laplacian with absorption,} J. Anal. Math. {125} (1),  113-138 (2015) 
\bibitem{O} F. Oliva, Regularizing effect of absorption terms in singular problems,  J. Math. Anal. Appl. {472} (1), 1136-1166 (2019)
  \bibitem{OP1} F. Oliva and F. Petitta,  {On singular elliptic equations with measure sources,} ESAIM Control Optim. Calc. Var. {22}  (1),  289-308  (2016) 
\bibitem{OP} F. Oliva and F. Petitta, {Finite and infinite energy solutions of singular elliptic problems: existence and uniqueness,}  J. Differential Equations {264} (1),  311-340 (2018) 
\bibitem{OsSe}
S. Osher and J. Sethian, Fronts propagating with curvature-dependent speed: algorithms based on Hamilton-Jacobi formulations. Journal of Computational Physics { 79} (1),  12-49 (1988)
\bibitem{ps} A. Porretta and S. Segura de Le\'on, 
Nonlinear elliptic equations having a gradient term with natural growth, 
J. Math. Pures Appl. (9) 85 (3), 465-492 (2006)
    \bibitem{Sapiro}  G. Sapiro,  {Geometric partial differential equations and image analysis,} Cambridge University Press, 2001.
\bibitem{ww} Y. Wang and  M. Wang,  Solutions to nonlinear elliptic equations with a gradient, 
Acta Math. Sci. Ser. B (Engl. Ed.) 35 (5), 1023-1036 (2015)
\bibitem{zw} W. Zhou, X. Wei, 
Some results on a singular parabolic equation in one dimension case, 
Math. Methods Appl. Sci. 36 (18), 2576-2587 (2013)
\end{thebibliography}
\end{document}